\documentclass[12pt]{amsart}
\usepackage{latexsym,amsfonts,amsmath,amssymb,amsthm,url}
\usepackage[english]{babel}
\usepackage[latin1]{inputenc}
\usepackage{amsbsy}
\usepackage{amscd}
\usepackage{graphicx,psfrag,epsfig}
\usepackage{enumerate}
\usepackage{dsfont}
\usepackage[normalem]{ulem}
\newtheorem{theorem}{Theorem}
\newtheorem{lemma}[theorem]{Lemma}
\newtheorem{proposition}[theorem]{Proposition}
\newtheorem{corollary}[theorem]{Corollary}

\newtheorem{remark}[theorem]{Remark}
\usepackage{color}
\definecolor{darkred}{rgb}{0.8,0,0}
\definecolor{darkblue}{rgb}{0,0,0.7}
\definecolor{darkgreen}{rgb}{0,0.4,0}

\newcommand{\RR}{\mathbb R}

\newcommand{\dd}{\;{\rm d}}
\newcommand{\YY}{{\mathcal Y}}
\newcommand{\KK}{{\mathcal K}}
\newcommand{\PP}{{\mathcal P}}
\newcommand{\LL}{{\rm L}}
\newcommand{\HH}{{\rm H}}
\newcommand{\rmW}{{\rm W}}
\newcommand{\EE}{{\mathcal E}}
\newcommand{\FF}{{\mathcal F}}
\newcommand{\WW}{{\mathcal W}}
\newcommand{\chls}{C_{\rm HLS}\,}
\newcommand{\R}{{\mathcal R}}
\newcommand{\DD}{{\mathcal D}}
\newcommand{\C}{{\mathcal C}}
\newcommand{\id}{{\rm id}}
\newcommand{\tr}{{\rm Tr}}
\renewcommand{\ln}{{\rm log}}

\begin{document}

\title[The parabolic-parabolic Keller-Segel system as a gradient flow]{The parabolic-parabolic Keller-Segel system with critical diffusion as a gradient flow in $\RR^d$, $d \ge 3$}

\author[A. Blanchet]{Adrien Blanchet$^1$}
\address{$^1$ TSE (GREMAQ, CNRS UMR~5604, INRA UMR~1291, Universit\'e de Toulouse), 21 All\'ee de Brienne, F--31000 Toulouse, France}
\email{Adrien.Blanchet@univ-tlse1.fr}

\author[Ph. Lauren\c cot]{Philippe Lauren\c cot$^2$}
\address{$^2$ Institut de Math\'ematiques de Toulouse, CNRS UMR~5219, Universit\'e de Tou\-lou\-se, F--31062 Toulouse cedex 9, France}
\email{laurenco@math.univ-toulouse.fr}

\date{\today}

\thanks{Partially supported by the project EVaMEF ANR-09-JCJC-0096-01.}

\begin{abstract} It is known that, for the parabolic-elliptic Keller-Segel system with critical porous-medium diffusion in dimension $\RR^d$, $d \ge 3$ (also referred to as the quasilinear Smoluchowski-Poisson equation), there is a critical value of the chemotactic sensitivity (measuring in some sense the strength of the drift term) above which there are solutions blowing up in finite time and below which all solutions are global in time. This global existence result is shown to remain true for the parabolic-parabolic Keller-Segel system with critical porous-medium type diffusion in dimension $\RR^d$, $d \ge 3$, when the chemotactic sensitivity is below the same critical value. The solution is constructed by using a minimising scheme involving the Kantorovich-Wasserstein metric for the first component and the $L^2$-norm for the second component. The cornerstone of the proof is the derivation of additional estimates which relies on a generalisation to a non-monotone functional of a method due to Matthes, McCann, \& Savar\'e (2009).
\end{abstract}

\keywords{chemotaxis; Keller-Segel model; degenerate diffusion; minimising scheme; Wasserstein distance}

\subjclass[2000]{35K65, 35K40, 47J30, 35Q92, 35B33}
\maketitle

\section{Introduction}
In space dimension $2$, the classical parabolic-parabolic Keller-Segel system reads~\cite{KS71}:
\begin{equation*}
\left\{
  \begin{array}{l}
\partial_t \rho= {\rm div} \left[ \nabla \rho - \chi_0 \,\rho\,\nabla c\right]\,, \vspace{.3cm}\\
\tau \,\partial_t c = D_0 \Delta c -\alpha_0\, c + \beta_0\, \rho\,,
  \end{array}
\quad (t,x) \in (0,\infty) \times \RR^2\,,
\right.
\end{equation*}
where $\rho\ge 0$ and $c\ge 0$ are the density of cells and the concentration of chemo-attractant, respectively, $\chi_0>0$, $D_0$, and $\beta_0$ are positive given constants, and $\tau$ and $\alpha_0$ are non-negative given constants. This system is one of the simplest models to describe the aggregation of cells by chemotaxis: the diffusion of the amoebae in a Petri dish is counterbalanced by the attraction toward higher concentrations of chemo-attractant that they themselves emit. This model has been widely studied mathematically in the last two decades with a main focus on the so-called parabolic-elliptic Keller-Segel system (also referred to as the Patlak-Keller-Segel system or the Smoluchowski-Poisson equation in astrophysics) which corresponds to the choice $\tau=0$, see~\cite{B3,Ho03,Ho04,P07} for a review. In particular, a striking feature of the Patlak-Keller-Segel system is that, given a non-negative and integrable initial condition $\rho_0$, there exists a global solution only if $\|\rho_0\|_1\le 8\pi\chi_0\beta_0/D_0$ while solutions blow in finite time if $\|\rho_0\|_1> 8\pi\chi_0\beta_0/D_0$. On the one hand, it is worth pointing out that this phenomenon has some biological and physical relevance as it corresponds to the formation of aggregates in biology or the gravitational collapse in astrophysics, see~\cite{B3,Ho03,Ho04,P07} and the references therein. On the other hand, let us also mention that such a behaviour is only observed in two space dimensions: indeed, solutions do exist globally in one space dimension whatever the value of $\|\rho_0\|_1$ while, in space dimension greater or equal to three, there are solutions blowing up in finite time emanating from initial data $\rho_0$ with arbitrary small mass $\|\rho_0\|_1$. 

It has been shown recently that, in higher space dimensions $d \ge 3$, a generalisation of this model, which is known in astrophysics as the generalised Smoluchowski-Poisson system~\cite{C02}, exhibits a similar threshold phenomenon. It differs from the above classical Keller-Segel system by a nonlinear diffusion and reads:
 \begin{equation}
  \label{PKS}
\left\{
  \begin{array}{l}
\partial_t \rho= {\rm div} \left[ \nabla \rho^m - \chi_0 \rho\nabla c\right]\,, \vspace{.3cm}\\
\tau \,\partial_t c = D_0 \Delta c -\alpha_0 \,c + \beta_0 \,\rho\,,
  \end{array}
\quad (t,x) \in (0,\infty) \times \RR^d\,,
\right.
\end{equation}
where 
\begin{equation*}
 m=2-\frac{2}{d} \in (1,2)\;,
\end{equation*}
the parameters $\chi_0>0$, $D_0$, and $\beta_0$ are positive given constants, and $\tau$ and $\alpha_0$ are non-negative given constants. With this choice of $m$, observe that the porous medium diffusion $\Delta \rho^m$ scales in the same way as the interaction term ${\rm div} [\rho\nabla c]$ and that we recover $m=1$ for $d=2$. Several results are now available for the parabolic-elliptic version of \eqref{PKS} corresponding to the choice $\tau=0$, see~\cite{BRB11,BCL09,S07b,S07,S06,SY11,SuT09a,SuT09b} and it was shown in \cite{BCL09,SuT09b} that there is a critical value $M_c$ of the mass $\|\rho_0\|_1$ of the initial condition $\rho_0$ depending upon $d$, $\chi_0$, $D_0$, and $\beta_0$ such that:
\begin{itemize}
\item if $\|\rho_0\|_1\le M_c$, then there is a global solution to~\eqref{PKS} with initial condition $\rho_0$,
\item given any $M>M_c$, there is at least one initial condition $\rho_0\ge 0$ with $\|\rho_0\|_1=M$ such that the corresponding solution to \eqref{PKS} blows up in finite time.
\end{itemize}
Though the well-posedness of the parabolic-parabolic Keller-Segel system~\eqref{PKS} ($\tau>0$) has been investigated recently, studies focused on the case where the exponent of the diffusion is strictly above $2(d-1)/d$~\cite{SeS06,S06} and nothing seems to be known when $m=2(d-1)/d$. The purpose of this paper is not only to show that global existence is true in that case for sufficiently small masses $\|\rho_0\|_1$ but also to give a quantitative estimate on the smallness condition related to a functional inequality. We actually prove global existence as soon as $\|\rho_0\|_1<M_c$, where $M_c$ is the already mentioned critical mass associated to the parabolic-elliptic Keller-Segel system.

Before describing more precisely our result and its proof, let us first get rid of inessential constants in \eqref{PKS} and normalise the initial condition $\rho_0$ for $\rho$ by introducing the rescaled functions
\begin{equation*}
\rho(t,x)=Ru(Tt,Xx)\quad\mbox{and}\quad c(t,x)=\Gamma v(Tt,Xx)\,, \qquad (t,x)\in [0,\infty)\times\RR^d\;,
\end{equation*}
with 
\begin{equation*}
R=D_0^{d/(d-2)},\;X=\frac{D_0^{1/(d-2)}}{\|\rho_0\|_{1}^d},\;T=\frac{D_0^{d/(d-2)}}{\|\rho_0\|_{1}^{2/d}}\;\mbox{ and } \;\Gamma=\beta_0 \|\rho_0\|_{1}^{2/d}\;.
\end{equation*}
Then $(u,v)$ solves
\begin{equation}
  \label{resPKS}
\left\{
  \begin{array}{l}
\partial_t u=  {\rm div} \left[ \nabla u^m - \chi  u\nabla v\right]\,, \vspace{.3cm}\\
\tau \partial_t v = \Delta v - \alpha\, v +  u\,,
  \end{array}
\quad (t,x) \in (0,\infty) \times \RR^d\,,  
\right.
\end{equation}
with
\begin{equation*}
\chi:=\frac{\chi_0}{D_0}\, \beta_0\,  \|\rho_0\|_{1}^{2/d}>0 \quad\mbox{ and }\quad \alpha:= \frac{\alpha_0}{D_0^{d/(d-2)}}\, \|\rho_0\|_{1}^{2/d}\ge 0\;,
\end{equation*}
while the initial condition $u_0$ of $u$ satisfies $\|u_0\|_1=1$. Owing to this transformation, the smallness condition for global existence on $\|\rho_0\|_1$ is equivalent to a smallness condition on $\chi$. 

Let us now describe the main result of this paper. We define
\begin{equation}
\chi_c:=\frac2{(m-1)\chls}\,, \label{critchi}
\end{equation}
 where $\chls$ is the constant of the variant of the Hardy-Littlewood-Sobolev inequality established in~\cite[Lemma 3.2]{BCL09}:
\begin{equation}\label{hls0}
  \chls:=\sup \left\{ \frac{\displaystyle{\int_{\RR^d} h(x)\,(\YY_0*h)(x)\dd x}}{\|h\|_m^m\,\|h\|_1^{2/d}} \,:\, h \in (\LL^1\cap \LL^m)(\RR^d), h \neq 0 \right\}< \infty\;.
\end{equation}
Here $\YY_0$ is the Poisson kernel, that is,
\begin{equation*}
\YY_0(x) := \int_0^\infty \frac1 {(4\pi s)^{d/2}} \exp \left(-\frac{|x|^2}{4s} \right)\dd s = c_d\ |x|^{2-d}\,, \quad x\in\RR^d\,, \end{equation*}
where $c_d:=\Gamma(d/2)/(2(d-2)\pi^{d/2})$.  It is shown in~\cite{BCL09} that, when $\tau=\alpha=0$, solutions to \eqref{resPKS} exist globally provided that $\chi\in (0,\chi_c]$ and the main purpose of this paper is to prove that, when $\tau>0$ and $\alpha\ge 0$, solutions to \eqref{resPKS} also exist globally in time if $\chi < \chi_c$.

\begin{theorem}[Global existence]\label{main}
Let $\tau>0$, $\alpha \ge 0$, $u_0$ be a non-negative function in $\LL^1(\RR^d,(1+|x|^2)\dd x)\cap \LL^m(\RR^d)$ satisfying $\|u_0\|_1=1$ and $v_0\in \HH^1(\RR^d)$. If $\chi < \chi_c$ then there exists a weak solution $(u,v)$ to the parabolic-parabolic Keller-Segel system~\eqref{resPKS}, that is, for all $t>0$ and $\xi\in\mathcal{C}_0^\infty(\RR^d)$,
\begin{itemize}
\item $u\in \LL^\infty(0,t;\LL^1(\RR^d,(1+|x|^2)\dd x)\cap \LL^m(\RR^d))$,  $u^{m/2} \in  \LL^2(0,t;\HH^1(\RR^d))$, 
\item $u(t)\ge 0$, $\|u(t)\|_1=1$,
\item $v \in \LL^\infty(0,t;\HH^1(\RR^d))\cap \LL^2(0,t;\HH^2(\RR^d)) \cap \rmW^{1,2}(0,t;\LL^2(\RR^d))$, $v(0)=v_0$,
\end{itemize}
and
\begin{align*}
\int_{\RR^d} \xi\ (u(t) - u_0)\dd x & + \int_0^t \int_{\RR^d} \left( \nabla (u^m) - \chi\ u\ \nabla v \right)\cdot \nabla\xi\dd x\dd s = 0\,, \\ 
\tau\, \partial_t v - \Delta v + \alpha\, v & = u \quad\mbox{ a.e. in }\quad (0,t)\times\RR^d\,.
\end{align*}
\end{theorem}
Note that Theorem~\ref{main} is valid without sign restriction on $v_0$. It is however easy to check, using the non-negativity of $u$ and the regularity of $v$, that $v_0 \ge 0$ implies $v \ge 0$. 
\begin{remark}
As for the classical parabolic-parabolic Keller-Segel system in space dimension $d=2$, we do not know what happens for $\chi\ge\chi_c$. It is worth mentioning that, unlike the parabolic-elliptic case, there might be global solutions to~\eqref{resPKS} for $\chi>\chi_c$ as recently shown in~\cite{BCD11} in the two dimensional case.
\end{remark}

\medskip

The proof relies on the fact that the parabolic-parabolic Keller-Segel system~\eqref{resPKS} can be seen as a gradient flow of the energy
\begin{equation}\label{eq:functional}
\EE_\alpha[u,v]:= \int_{\RR^d} \left\{\frac{|u(x)|^m}{\chi(m-1)}-u(x)\,v(x) + \frac12\,|\nabla v(x)|^2+\frac{\alpha}{2}\,v(x)^2 \right\}\dd x \;,
\end{equation}
in $\PP_2(\RR^d)\times\LL^2(\RR^d)$ endowed with the Kantorovich-Wasserstein metric for the first component and the usual $\LL^2$-norm for the second component, where $\PP_2(\RR^d)$ is the set of probability measures on $\RR^d$ with finite second moment. Let us mention at this point that, since the pioneering works~\cite{JKO98,Ot01}, several equations have been interpreted as a gradient flow for a Wasserstein distance, see, e.g.,~\cite{AGS08,AS08,CG04,MMS09,Vi03}, including the parabolic-elliptic Keller-Segel system in two space dimensions which actually can be seen as a non-local partial differential equation and thus handled as a single equation~\cite{BCC08,BCC12}, and the literature concerning this issue is steadily growing. In contrast, there are only a few examples of systems which can be interpreted as gradient flows with respect to some metric involving a Wasserstein distance and, as far as we know, a minimising scheme is used to construct solutions to the parabolic-parabolic Keller-Segel system (with linear diffusion $m=1$) in two space dimensions~\cite{CLxx} and to a thin film approximation of the Muskat problem~\cite{LMxx}.

\medskip

To describe more precisely the approach used herein, we introduce the set
\begin{equation*}
\KK:=(\PP_2\cap \LL^m)(\RR^d) \times \HH^1(\RR^d)
\end{equation*}
on which the energy $\EE_\alpha$ is well-defined, see below, and we construct solutions to \eqref{resPKS} by the now classical minimising scheme: given an initial condition $(u_0,v_0)\in\KK$ and a time step $h>0$, we define a sequence $(u_{h,n},v_{h,n})_{n\ge 0}$ in $\KK$ as follows:
\begin{equation}\label{scheme:jko}
\left\{
  \begin{array}{l}
    (u_{h,0},v_{h,0}) = (u_0,v_0)\,, \vspace{.3cm}\\
\displaystyle(u_{h,n+1},v_{h,n+1})  \in {\rm Argmin}_{(u,v)\in\KK}  \mathcal{F}_{h,n}[u,v]\,, \qquad n\ge 0\,,
  \end{array}
\right.
\end{equation}
where 
\begin{equation*}
   \mathcal{F}_{h,n}[u,v] := \frac{1}{2h}\ \left[ \frac{\WW_2^2(u,u_{h,n})}{\chi} + \tau\ \|v-v_{h,n}\|_2^2 \right] + \EE_\alpha[u,v]\,,
\end{equation*}
and $\WW_2$ is the Kantorovich-Wasserstein distance on $\PP_2(\RR^d)$. Several difficulties arise in the proof of the well-posedness and convergence of the previous minimising scheme. First, as the energy $\EE_\alpha$ is not displacement convex, standard results from~\cite{AGS08,Vi03} do not apply and even the existence of a minimiser is not clear. Nevertheless, the assumption $\chi<\chi_c$ and a further development of the modified Hardy-Littlewood-Sobolev inequality~\eqref{hls0} allow us to obtain an $(\LL^1 \cap \LL^m)(\RR^d) \times \HH^1(\RR^d)$ bound on minimising sequences which permits in particular to pass to the limit in the term in $\EE_\alpha[u,v]$ involving the product $uv$, see Proposition~\ref{pr:b4}, and prove the existence of a minimiser. A similar problem is faced when using this approach to construct global solutions to the parabolic-elliptic Keller-Segel system in two space dimensions~\cite{BCC08,BCC12} and is solved there with the help of the logarithmic Hardy-Littlewood-Sobolev inequality. A second difficulty stems from the fact that, unlike the parabolic-elliptic Keller-Segel system, the Cauchy problem~\eqref{resPKS} cannot be reduced to a single equation. More precisely, to obtain the Euler-Lagrange equation satisfied by a minimiser $(\bar{u},\bar{v})$ of $\mathcal{F}_{h,n}$ in $\KK$, the parameters $h$ and $n$ being fixed, we consider a ``horizontal'' perturbation for $\bar{u}$ and a $\LL^2$-perturbation for $\bar{v}$ defined for $\delta\in (0,1)$ by 
\begin{equation*}
 u_\delta= (\id + \delta\,\zeta) \# \bar{u}\;, \quad v_\delta := \bar{v} + \delta\, w \;,
\end{equation*}
where $\zeta\in\C^\infty_0(\RR^d;\RR^d)$ and $w\in\C^\infty_0(\RR^d)$ are two smooth test functions, $\id$ is the identity function of $\RR^d$, and $T\#\mu$ is the push-forward measure of the measure $\mu$ by the map $T$. Identifying the Euler-Lagrange equation requires to pass to the limit as $\delta\to 0$ in 
\begin{equation*}
\frac{\WW^2_2(u_\delta ,u_{h,n})-\WW^2_2(\bar{u} ,u_{h,n})}{2\delta} \;\;\mbox{ and }\;\; \frac{\|u_\delta\|^m_m-\|\bar{u}\|^m_m}{\delta}\,,
\end{equation*}
which can be performed by standard arguments \cite{AGS08,Vi03}, but also in 
\begin{equation*}
\frac{1}{\delta}\ \int_{\RR^d} (\bar{u}\,\bar{v}-u_\delta\,v_\delta)(x) \dd x = \int_{\RR^d} \bar{u}(x) \left[ \frac{\bar{v}(x)-\bar{v}(x+\delta\zeta(x))}{\delta} - w(x+\delta\zeta(x)) \right]\dd x\;.
\end{equation*}
This is where the main difficulty lies since 
\begin{equation*}
\frac{\bar{v}\!\circ\!(\id + \delta\zeta) -\bar{v}}{\delta} \rightharpoonup \zeta \cdot \nabla\bar{v}\quad \mbox{in $\LL^2(\RR^d)$,} 
\end{equation*}
but $\bar{u}$ is only in $(\LL^1 \cap \LL^m)(\RR^d)$ and $m<2$. Therefore, the regularity of $(\bar{u},\bar{v})$ is not enough to pass to the limit in this term and has to be improved. To this end, a powerful technique is developed in~\cite{MMS09}, the main idea being to use as test functions in the minimising scheme the solution to a suitably chosen simpler gradient flow and analyse the short time behaviour of the outcoming inequality. In the framework considered in~\cite{MMS09} (see also~\cite{BCC12}), the choice of the auxiliary gradient flow naturally comes from the existence of another Liapunov functional which is different from the energy. Such a nice structure does not seem to be available here but it turns out that we are able to somehow extend the scope of this method by finding a simple gradient flow $(s\mapsto (U(s),V(s))$ such that $s\mapsto \EE_\alpha[U(s),V(s)]$ is the sum of a decreasing function and a remainder which can be controlled: for the problem \eqref{resPKS} under study, the choice is the solutions $U$ and $V$ to the initial value problems
\begin{equation*}
  \partial_s U - \Delta U=0\quad\mbox{in $(0,\infty) \times \RR^d$,}\qquad U(0)=\bar{u}\;,
\end{equation*}
and
\begin{equation*}
  \partial_s V - \Delta V + \alpha V=0\quad\mbox{in $(0,\infty) \times \RR^d$,}\qquad V(0)=\bar{v}\;.
\end{equation*}
Owing to the properties of $(U,V)$, we have $\EE_\alpha[\bar{u},\bar{v}]\le \EE_\alpha[U(s),V(s)]$ for all $s\ge 0$ and this particular choice and the known regularity of $(U,V)$ allow us to pass to the limit as $s\to 0$ in the inequality $(\EE_\alpha[U(s),V(s)] - \EE_\alpha[\bar{u},\bar{v}])/s\ge 0$ and finally obtain the desired $\LL^2(\RR^d)$-regularity of $\bar{u}$, see Proposition~\ref{pr:b5} and Corollary~\ref{cor:b6}. This regularity allows us to pass to the limit in the Euler-Lagrange equation. This analysis is performed  in Section~\ref{sec:jko} where we prove the well-posedness of the minimising scheme~\eqref{scheme:jko} and study the properties of the minimisers. Section~\ref{sec:conv} is dedicated to the proof of the convergence of the scheme as $h\to 0$, from which Theorem~\ref{main} follows. 

\section{The minimising scheme}\label{sec:jko}

For $\alpha\ge 0$ and $(u,v) \in (\LL^1\cap \LL^m)(\RR^d) \times \HH^1(\RR^d)$, we define the functional $\EE_\alpha[u,v]$ by~\eqref{eq:functional}. First notice that it is well-defined as the integrability properties of $u$ and $v$ ensure that $uv\in \LL^1(\RR^d)$: indeed, by the H\"older and Sobolev inequalities
\begin{equation}\label{220}
\|uv\|_1 \le \|u\|_{2d/(d+2)}\ \|v\|_{2d/(d-2)} \le C\ \|u\|_m^{m/2}\ \|u\|_1^{1/d}\ \|\nabla v\|_2 < \infty\,.
\end{equation}

We next study the properties of $\EE_\alpha$, $\alpha\ge 0$. To this end, let $\YY_\alpha$ be the Bessel kernel defined for $\alpha\ge 0$ by:
\begin{equation*}
  \YY_\alpha(x):=\int_0^\infty \frac1 {(4\pi s)^{d/2}} \exp \left(-\frac{|x|^2}{4s}-\alpha s \right)\dd s\,, \quad x\in\RR^d\,,
\end{equation*}
the case $\alpha=0$ corresponding to the already defined Poisson kernel. For $u\in\LL^1(\RR^d)$, $S_\alpha(u):=\YY_\alpha*u$ solves
\begin{equation}\label{alphas}
  -\Delta S_\alpha (u)+\alpha S_\alpha(u)=u\quad\mbox{in $\RR^d$}
\end{equation}
in the sense of distributions, see~\cite[Theorem~6.23]{LL01}. The Bessel kernel is also referred to as the screened Poisson or Yukawa potential in the literature. 

We now show that a modified Hardy-Littlewood-Sobolev inequality is valid for the Bessel kernel $\YY_\alpha$ for $\alpha>0$:
\begin{lemma}[Hardy-Littlewood-Sobolev inequality for the Bessel kernel]\label{inequality} For $\alpha>0$,
  \begin{equation}
    \label{eq:hlspoisson}
     \sup \left\{ \frac{\displaystyle{\int_{\RR^d} h(x)\,(\YY_\alpha*h)(x)\dd x}}{\|h\|_m^m\,\|h\|_1^{2/d}}\,:\, h \in (\LL^1 \cap \LL^m)(\RR^d), h \neq 0 \right\}=\chls\,,
  \end{equation}
where $\chls$ is defined in \eqref{hls0}.
\end{lemma}
\begin{proof}
We denote the left-hand side of \eqref{eq:hlspoisson} by $\mu_\alpha$. By the definition of $\YY_\alpha$ and $\YY_0$,
\begin{equation}
  0<\YY_\alpha(x) \le \YY_0(x) \;\;\mbox{ for }\;\; x\in\RR^d\;. \label{b12}
\end{equation}
Therefore, for $h \in (\LL^1\cap \LL^m)(\RR^d)$,
\begin{multline*}
  \int_{\RR^d} h(x)\,(\YY_\alpha*h)(x)\dd x \le \int_{\RR^d} |h(x)|\,(\YY_\alpha*|h|)(x)\dd x \le \int_{\RR^d} |h(x)|\,(\YY_0*|h|)(x)\dd x \\\le \chls \|h\|_m^m\,\|h\|_1^{2/d}\;,
\end{multline*}
whence $\mu_\alpha$ is finite with $\mu_\alpha \le \chls $.

Conversely, given a non-negative function $h \in (\LL^1 \cap \LL^m)(\RR^d)$, we define $h_\lambda(x)=\lambda^dh(\lambda x)$ for $\lambda >1$ and $x \in \RR^d$. Observe that
\begin{multline}
  \label{eq:asterix}
  \|h_\lambda\|_1=\|h\|_1,\qquad \|h_\lambda\|_m^m=\lambda^{d-2}\,\|h\|_m^m\\
\mbox{and}\,\int_{\RR^d} h_\lambda(x)\,(\YY_\alpha*h_\lambda)(x)\dd x = \lambda^{d-2}\int_{\RR^d} h(x)\,(\YY_{\alpha\lambda^{-2}}*h)(x)\dd x\;.
\end{multline}
As $\lambda \to \infty$, $\YY_{\alpha\lambda^{-2}}$ converges pointwisely to $\YY_{0}$ for $x \neq 0$. Moreover, by~\eqref{b12},
$$0 \le h(x)\,(\YY_{\alpha\lambda^{-2}}*h)(x) \le h(x)\,(\YY_0*h)(x)\,, \qquad x\in\RR^d\,.$$
Since $h\, (\YY_0*h)\in \LL^1(\RR^d)$ by the classical Hardy-Littlewood-Sobolev inequality~\cite[Theorem~4.3]{LL01}, we infer from Lebesgue's dominated convergence theorem that
\begin{equation*}
 \lim_{\lambda \to \infty} \int_{\RR^d}h(x)\,(\YY_{\alpha\lambda^{-2}}*h)(x) \dd x =\int_{\RR^d}h(x)\,(\YY_{0}*h)(x) \dd x\;.
\end{equation*}
By~\eqref{eq:asterix}, it implies
\begin{eqnarray*}
\mu_\alpha \ge \lim_{\lambda \to \infty}\frac{\displaystyle{\int_{\RR^d} h_\lambda(x)\,(\YY_\alpha*h_\lambda)(x)\dd x}}{\|h_\lambda\|_m^m\,\|h_\lambda\|_1^{2/d}} & = & \lim_{\lambda \to \infty}\frac{\displaystyle{\int_{\RR^d} h(x)\,(\YY_{\alpha\lambda^{-2}}*h)(x)\dd x}}{\|h\|_m^m\,\|h\|_1^{2/d}} \\
& = &\frac{\displaystyle{\int_{\RR^d} h(x)\,(\YY_{0}*h)(x)\dd x}}{\|h\|_m^m\,\|h\|_1^{2/d}}\;.
\end{eqnarray*}
This inequality being valid for all non-negative $h \in (\LL^1\cap \LL^m)(\RR^d)$, it readily extends to all $h \in (\LL^1 \cap \LL^m)(\RR^d)$ so that $\mu_\alpha \ge \chls $, and the proof is complete.
\end{proof}

With the above notations, we have an alternative formula for $\EE_\alpha$ for $\alpha>0$.
\begin{lemma}\label{below}
 Let $\alpha>0$, $u \in (\LL^1\cap\LL^m)(\RR^d)$, and $v \in \HH^1(\RR^d)$. Then
 \begin{equation*}
   \EE_\alpha[u,v] = \EE_\alpha[u,S_\alpha(u)]+\frac12 \|\nabla \left(v - S_\alpha(u)\right)\|^2_2+\frac{\alpha}{2}\, \|v - S_\alpha(u)\|^2_2\,.
 \end{equation*}
\end{lemma}
\begin{proof}
We proceed as in~\cite{CC10} and first claim that $S_\alpha(u)\in\HH^1(\RR^d)$ so that $\EE_\alpha[u,S_\alpha(u)]$ is well defined. Indeed, since $m\in (1,d/2)$ and $u\in (\LL^1\cap \LL^m)(\RR^d)$, we infer from \cite[Theorem~6.23~(v)]{LL01} that $S_\alpha(u)\in\LL^r(\RR^d)$ for all $r\in [1,dm/(d-2m)]$. In addition, using a rescaling with respect to the parameter $\alpha$, we deduce from \cite[Chapter~V, \S~3.3, Theorem~3]{St70} that, for all $p\in (1,m]$, there is $C=C(d,p)$ such that
\begin{equation}
\alpha \|S_\alpha(u)\|_p + \sqrt{\alpha}\ \|\nabla S_\alpha(u)\|_p + \|D^2 S_\alpha(u)\|_p \le C(p)\ \|u\|_p\,. \label{estbessel}
\end{equation}
In particular, since $2d/(d+2)\in (1,m)$, we have $S_\alpha(u)\in \rmW^{2,2d/(d+2)}(\RR^d)$ and the Sobolev embedding guarantees that $\nabla S_\alpha(u)\in \LL^2(\RR^d)$. Also, $dm/(d-2m)>2$ and we thus have $S_\alpha(u)\in \LL^2(\RR^d)$, thereby completing the proof of the claim.

We now compute
\begin{align*}
  \EE_\alpha[u,v] - \EE_\alpha[u,S_\alpha(u)]=&
\int_{\RR^d}\Big\{-u\,[v-S_\alpha(u)]+ \frac12\,\left|\nabla[v-S_\alpha(u)]\right|^2 \\
&\qquad+\nabla[v-S_\alpha(u)]\cdot  \nabla S_\alpha(u)+\frac{\alpha}2\,|v-S_\alpha(u)|^2\\
&\qquad+\alpha\, [v-S_\alpha(u)]\,S_\alpha(u) \Big\}\dd x\\
=&\,\frac12\,\|\nabla[v-S_\alpha(u)]\|^2_2+\frac{\alpha}2\,\|v-S_\alpha(u)\|^2_2\\
&\qquad+ \int_{\RR^d}  [v-S_\alpha(u)]\left[-\Delta S_\alpha(u)+\alpha S_\alpha(u)-u) \right]\dd x
\end{align*}
which gives the stated result by using the Bessel equation~\eqref{alphas}.
\end{proof}
A lower bound on $\EE_\alpha$ follows from Lemmas~\ref{inequality} and~\ref{below}:
\begin{lemma}[Lower bound on $\EE_\alpha$]\label{lem:213}
Let $\chi_c$ be defined in~\eqref{critchi}. If $\alpha>0$ and $(u,v) \in (\LL^1\cap\LL^m)(\RR^d)\times \HH^1(\RR^d)$, then
  \begin{equation}
    \label{eq:e}\tag{i}
    \EE_\alpha[u,v] \ge \frac{\chls}{2\chi}\ \left( \chi_c - \chi\ \|u\|_1^{2/d} \right)\,\|u\|^m_m\;.
  \end{equation}
Moreover there exists $C_1>0$ depending only on $d$ such that
\begin{equation}
  \label{eq:h}\tag{ii}
  \|\nabla v\|^2_2 + \alpha \|v\|^2_2 \le 4\ \EE_\alpha[u,v] + C_1\ \|u\|_1^{2/d}\ \|u\|_m^{m}\;.
\end{equation}
\end{lemma}
\begin{proof}
Owing to the property~\eqref{alphas} of the Bessel potential and the modified Hardy-Littlewood-Sobolev inequality for the Bessel kernel~\eqref{eq:hlspoisson}, we have:
  \begin{align*}
    \EE_\alpha[u,S_\alpha(u)] &= \frac{\|u\|^m_m}{\chi (m-1)}-\frac{1}{2}\int_{\RR^d}u\,S_\alpha(u)\dd x\\
&\ge \frac{\|u\|^m_m}{\chi (m-1)} - \frac{\chls}{2} \,\|u\|^m_m\,\|u\|^{2/d}_1\\
&\ge \frac{\chls}{2\chi}\ \left( \chi_c - \chi\ \|u\|_1^{2/d} \right)\,\|u\|^m_m\;.
  \end{align*}
The estimate~\eqref{eq:e} then readily follows from Lemma~\ref{below} and the above inequality.

Next, by~\eqref{220} and Young's inequality, we have
\begin{align*}
\|\nabla v\|^2_2 + \alpha \|v\|^2_2  &\le 2\EE_\alpha[u,v]+2\|u\,v\|_1\\
&\le 2\EE_\alpha[u,v]+C\ \|u\|_m^{m/2}\ \|u\|_1^{1/d}\ \|\nabla v\|_2 \\
&\le 2\EE_\alpha[u,v]+\frac{\|\nabla v\|^2_2}{2} + C\ \|u\|_1^{2/d}\ \|u\|_m^m \\
&\le 2\EE_\alpha[u,v]+\frac{\|\nabla v\|^2_2 + \alpha \|v\|^2_2}{2} + C\ \|u\|_1^{2/d}\ \|u\|_m^m\;,
\end{align*}
which gives the stated result~\eqref{eq:h}.
\end{proof}

Though Lemma~\ref{below} is not true for $\alpha=0$ since $S_0(u)\not\in \HH^1(\RR^d)$, it turns out that Lemma~\ref{lem:213} is still valid in that case as we show now.

\begin{lemma}[Lower bound on $\EE_0$]\label{lem:213b}
Let $\chi_c$ be defined in~\eqref{critchi}. If $(u,v) \in (\LL^1\cap\LL^m)(\RR^d)\times \HH^1(\RR^d)$, then the statement of Lemma~\ref{lem:213} remains true for $\alpha=0$.
\end{lemma}
\begin{proof}
Since we cannot use $(u,S_0(u))$ in $\EE_0$, we use an approximation argument: recalling that, given $\varepsilon\in (0,1)$, $S_\varepsilon(u)\in\HH^1(\RR^d)$, we may compute
$$
\EE_0[u,v] - \EE_0[u,S_\varepsilon(u)] = \frac{1}{2}\ \|\nabla(v-S_\varepsilon(u))\|_2^2 - \varepsilon\ \int_{\RR^d} S_\varepsilon(u)\ (v-S_\varepsilon(u)) \dd x \,.
$$
We infer from H\"older's, Sobolev's and Young's inequalities that 
\begin{align*}
\EE_0[u,v] & - \EE_0[u,S_\varepsilon(u)] \\
& \ge \frac{1}{2}\ \|\nabla(v-S_\varepsilon(u))\|_2^2 - \varepsilon\ \| S_\varepsilon(u)\|_{2d/(d+2)}\ \|(v-S_\varepsilon(u))\|_{2d/(d-2)} \\
& \ge \frac{1}{2}\ \|\nabla(v-S_\varepsilon(u))\|_2^2 - C\ \varepsilon\ \| S_\varepsilon(u)\|_{2d/(d+2)}\ \|\nabla (v-S_\varepsilon(u))\|_2\\
& \ge -C\ \varepsilon^2\ \| S_\varepsilon(u)\|_{2d/(d+2)}^2\,,
\end{align*}
whence
\begin{equation}
\EE_0[u,v] \ge \EE_0[u,S_\varepsilon(u)] - C\ \varepsilon^2\ \| S_\varepsilon(u)\|_{2d/(d+2)}^2\,. \label{suppl1}
\end{equation}
It next follows from the modified Hardy-Littlewood-Sobolev inequality \eqref{eq:hlspoisson} that 
\begin{align*}
\EE_0[u,S_\varepsilon(u)] & = \frac{\|u\|_m^m}{\chi(m-1)} - \frac{1}{2}\ \int_{\RR^d} S_\varepsilon(u)\ \left( 2u + \Delta S_\varepsilon(u) \right) \dd x \\
& = \frac{\|u\|_m^m}{\chi(m-1)} - \frac{1}{2}\ \int_{\RR^d} S_\varepsilon(u)\ \left( u + \varepsilon\ S_\varepsilon(u) \right) \dd x \\
& \ge \frac{\|u\|_m^m}{\chi(m-1)} - \frac{\chls}{2}\ \|u\|_m^m\ \|u\|_1^{2/d} - \frac{\varepsilon}{2}\ \|S_\varepsilon(u)\|_2^2 \\
& \ge \frac{\chls}{2\chi}\ \left( \chi_c - \chi\ \|u\|_1^{2/d} \right)\ \|u\|_m^m - \frac{\varepsilon}{2}\ \|S_\varepsilon(u)\|_2^2\,.
\end{align*}
Combining \eqref{suppl1} and the above inequality gives
\begin{equation}
\EE_0[u,v] \ge \frac{\chls}{2\chi}\ \left( \chi_c - \chi\ \|u\|_1^{2/d} \right)\ \|u\|_m^m - C\ \varepsilon^2\ \| S_\varepsilon(u)\|_{2d/(d+2)}^2 - \frac{\varepsilon}{2}\ \|S_\varepsilon(u)\|_2^2\,. \label{suppl2}
\end{equation}
We are then left to study the behaviour of the last two terms of the right-hand side of \eqref{suppl2} as $\varepsilon\to 0$. To this end, we note that $S_\varepsilon(u) = \YY_0*(u-\varepsilon\ S_\varepsilon(u))$ and it follows from the regularity of solutions to the Poisson equation (see, e.g., \cite[Theorem~10.2]{LL01}) and the estimates \eqref{estbessel} for the Bessel potential that, for $p\in (1,m]$, 
\begin{equation}
\|S_\varepsilon(u)\|_{dp/(d-2p)} \le C(p)\ \|u-\varepsilon\ S_\varepsilon(u)\|_p \le C(p)\ \|u\|_p\,.
\end{equation}
We set $p_m:=dm/(d-2m)>1$ and consider $p_0\in (1,2d/(d+2))$ to be determined later on. Since $2\in (p_0,p_m)$ and $2d/(d+2) \in (p_0,p_m)$, we have
$$
\lambda := \frac{2-p_0}{p_m-p_0} \in (0,1)\,, \qquad \mu:= \frac{2d -(d+2)p_0}{(d+2)(p_m-p_0)} \in (0,1)\,, 
$$
and we infer from H\"older's inequality, \eqref{estbessel}, and \eqref{suppl2} that
\begin{align}
\varepsilon\ \|S_\varepsilon(u)\|_2^2 & \le \varepsilon\ \|S_\varepsilon(u)\|_{p_m}^{\lambda p_m}\ \|S_\varepsilon(u)\|_{p_0}^{(1-\lambda)p_0} \nonumber\\ 
& \le C(p_0)\ \varepsilon^{1+(\lambda-1)p_0}\ \|u\|_m^{\lambda p_m}\ \|u\|_{p_0}^{(1-\lambda)p_0}\,, \label{suppl3}
\end{align}
and
\begin{align}
\quad \varepsilon^2\ \|S_\varepsilon(u)\|_{2d/(d+2)}^2 & \le \varepsilon^2\ \|S_\varepsilon(u)\|_{p_m}^{(d+2) \mu p_m/d}\ \|S_\varepsilon(u)\|_{p_0}^{(d+2)(1-\mu)p_0/d} \nonumber\\ 
& \le C(p_0)\ \varepsilon^{(2d+(d+2)(\mu-1)p_0)/d}\ \|u\|_m^{(d+2)\mu p_m/d}\ \|u\|_{p_0}^{(d+2)(1-\mu)p_0/d}\,. \label{suppl4}
\end{align}
Since
$$
1+(\lambda-1) p_0 = \frac{p_m + p_0 - p_m\, p_0}{p_m-p_0} \quad\mathop{\longrightarrow}_{p_0\to 1}\quad \frac{1}{p_m-1}>0\,, $$
and
$$
2d+(d+2)(\mu-1)p_0 = \frac{p_m\, (2d - (d+2)\, p_0)}{p_m-p_0} \quad\mathop{\longrightarrow}_{p_0\to 1}\quad \frac{(d-2)\, p_m}{p_m-1}>0\,,
$$
we can find $p_0\in (1,2d/(d+2))$ sufficiently close to $1$ such that $1+(\lambda-1) p_0>0$ and $2d+(d+2)(\mu-1)p_0>0$. Owing to \eqref{suppl3} and \eqref{suppl4}, this property readily implies that the last two terms of the right-hand side of \eqref{suppl2} converges to zero as $\varepsilon\to 0$. Consequently, letting $\varepsilon\to 0$ in \eqref{suppl2} gives that the estimate \eqref{eq:e} of Lemma~\ref{lem:213} is also true for $\alpha=0$. The proof of the estimate for $\|\nabla v\|_2$ is then the same as that of Lemma~\ref{lem:213}~\eqref{eq:h}.
\end{proof}

After this preparation, we are in a position to define and study the minimisation scheme. Set
\begin{equation*}
 \KK:=(\PP_2\cap \LL^m)(\RR^d) \times \HH^1(\RR^d),
\end{equation*}
where $\PP_2(\RR^d)$ is the set of probability measures on $\RR^d$ with finite second moment. For $\alpha\ge 0$, $h \in (0,1)$ and $(u_0,v_0) \in \KK$, we define the functional
\begin{equation*}
  \FF_h[u,v]:=\frac{1}{2h} \left(\frac{\WW^2_2(u,u_0)}{\chi}+\tau\|v-v_0\|^2_2\right) + \EE_\alpha[u,v]\;,
\end{equation*}
where $\WW_2$ is the Kantorovich-Wasserstein distance on $\PP_2(\RR^d)$. Our aim is to minimise $\FF_h$ on $\KK$.

\subsection{Existence and properties of minimisers}
\begin{proposition}[Existence of minimisers]\label{pr:b4}
Given $\chi \in (0,\chi_c)$, $(u_0,v_0) \in \KK$, and $h\in (0,1)$, there exists at least a minimiser $(u,v) \in \KK$ of $\FF_h$ in $\KK$. Moreover, any minimiser $(u,v)$ of $\FF_h$ in $\KK$ satisfies
\begin{equation}
    \label{eq:214}
    \FF_h[u,v]\le \FF_h[u_0,v_0]=\EE_\alpha[u_0,v_0]\;.
\end{equation}
\end{proposition}
\begin{proof} We proceed in two steps.

\noindent\textbf{Step~1: Estimates.} Consider $(u,v)\in\KK$. The lower bound on the free energy, see Lemma~\ref{lem:213} (i), ensures that
 \begin{equation}\label{216}
   \FF_h[u,v] \ge \frac1{2h} \left(\frac{\WW_2^2(u,u_0)}{\chi}+ \tau \|v-v_0\|^2_2 \right) + \eta \|u\|^m_m\,,
 \end{equation}
where $\eta:=\chls (\chi_c-\chi)/2\chi>0$. Therefore, $\FF_h$ is non-negative in $\KK$ and we may define
\begin{equation}\label{eq:217}
  \omega:= \inf_{(u,v) \in \KK} \FF_h[u,v]\ge 0\;.
\end{equation}
Let $(u_k,v_k)_{k\ge 1}$ be a minimising sequence with $\omega \le \FF_h[u_k,v_k] \le \omega + 1/k$ for $k\ge 1$. Recalling that
\begin{equation*}
  \WW_2(\delta_0,\mu)\le \WW_2(\delta_0,u_0) + \WW_2(u_0,\mu)\,, \qquad \mu\in\PP_2(\RR^d)\,,
\end{equation*}
where $\delta_0$ denotes the Dirac mass in $\RR^d$ centred at $x=0$, we deduce from \eqref{216} that, for $k\ge 1$, 
\begin{align}
  \int_{\RR^d} |x|^2 u_k(x)\dd x & \le 2 \int_{\RR^d} |x|^2 u_0(x)\dd x + 2 \WW_2^2(u_k,u_0) \nonumber\\
  & \le 2 \int_{\RR^d} |x|^2 u_0(x)\dd x + 4h\chi \FF_h[u_k,v_k] \nonumber\\
  & \le 2 \int_{\RR^d} |x|^2 u_0(x)\dd x + 4h\chi (\omega+1)\;. \label{b19b}
\end{align}
It also follows from \eqref{216} and \eqref{eq:217} that, for $k\ge 1$, 
  \begin{equation}\label{b19}
\|u_k\|^m_m + \|v_k-v_0\|_2^2 \le \frac{\FF_h[u_k,v_k]}{\eta} + \frac{2h \FF_h[u_k,v_k]}{\tau} \le \left( \frac{1}{\eta} + \frac{2h}{\tau} \right)\ (\omega + 1)\,.
  \end{equation}
Furthermore, we infer from Lemma~\ref{lem:213} (ii) that
\begin{equation}\label{eq:221}
\|\nabla v_k\|^2_2 \le 4\ \EE_\alpha[u_k,v_k] + C_1\ \|u_k\|_m^m \le C\,,\quad k\ge 1\;.
\end{equation}

\noindent\textbf{Step~2: Passing to the limit.}  Owing to~\eqref{b19b}, \eqref{b19}, and~\eqref{eq:221}, it follows from the Dunford-Pettis theorem that there are $(u_\infty,v_\infty) \in (\LL^1\cap \LL^m)(\RR^d) \times \HH^1(\RR^d)$ and a subsequence of $(u_k,v_k)_k$, which is not relabelled, such that
\begin{align}
  u_k &\rightharpoonup u_\infty \quad\mbox{in $(\LL^1\cap \LL^m)(\RR^d)$}\label{eq:222}\\
  v_k &\rightharpoonup v_\infty \quad\mbox{in $\HH^1(\RR^d)$.}\label{eq:223}
\end{align}
Since $u_k$ is a probability measure for all $k \ge 1$, the convergence~\eqref{eq:222} guarantees that $u_\infty$ is a probability measure. Next, a classical truncation argument and \eqref{b19b} imply that the second moment of $u_\infty$ is finite. Therefore, $u_\infty$ belongs to $\PP_2(\RR^d)$ and $(u_\infty,v_\infty) \in \KK$. 

Next, weak lower semicontinuity arguments and the properties of the Kantorovich-Wasserstein distance allow us to deduce from~\eqref{eq:222} and~\eqref{eq:223} that
\begin{multline}
  \label{eq:224}
  \frac{1}{2h} \left(\frac{\WW_2^2(u_\infty,u_0)}{\chi}+ \tau \|v_\infty-v_0\|^2_2 \right) + \frac{\|u_\infty\|^m_m}{\chi(m-1)} + \frac{\|\nabla v_\infty\|^{2}_2}{2} + \frac{\alpha\,\|v_\infty\|^{2}_2}{2} \\\le \liminf_{k \to \infty} \left(  \FF_h[u_k,v_k]  + \int_{\RR^d}u_k(x)\,v_k(x)\dd x\right)\;.
\end{multline}

We are then left with passing to the limit in the last term of the right-hand side of~\eqref{eq:224}. For that purpose, given $n \ge 1$, we fix a truncation function $\theta_n \in  \C^\infty_0(\RR^d)$ satisfying $0\le\theta_n(x)\le 1$ for all $x \in \RR^d$ and
\begin{equation*}
\theta_n(x):=
\left\{
  \begin{array}{ll}
1&\quad \mbox{if $|x| \le n$}\\
0&\quad \mbox{if $|x| \ge 2n$}\;.
  \end{array}
\right. 
\end{equation*}
By~\eqref{eq:221}, $(\theta_n\,v_k)_{k\ge 1}$ is bounded in $\HH^1_0(B(0,2n))$ for each $n\ge 1$ and, using a diagonal process, we can extract a further subsequence, not relabelled, of $(v_k)_{k\ge 1}$ such that
\begin{equation}
  \label{eq:225}
  \theta_n \,v_k \rightarrow \theta_n \,v_\infty\quad\mbox{in $\LL^p(\RR^d)$ for any $p\in [2,2d/(d-2))$ and $n\ge 1$.}
\end{equation}
It next follows from the H\"older and Sobolev inequalities that
\begin{align}
  \left| \int_{\RR^d}(1-\theta_n)\,u_k\,v_k\dd x\right| &\le \|v_k\|_{2d/(d-2)}\ \left[\int_{\{|x|\ge n\}} u_k^{2d/(d+2)} \dd x \right]^{(d+2)/2d}\notag\\
&\le C\ \|\nabla v_k\|_{2}\,\|u_k\|^{m/2}_{m}\ \left[\int_{\{|x|\ge n\}}u_k\dd x\right]^{1/d}\notag\\
&\le \frac{C}{n^{2/d}}\  \|\nabla v_k\|_{2}\,\|u_k\|^{m/2}_{m}\ \left[\int_{\RR^d}|x|^2u_k\dd x\right]^{1/d}\le \frac{C}{n^{2/d}} \label{eq:226}\;,
\end{align}
the last inequality being a consequence of~\eqref{b19b} and~\eqref{b19}.
Similarly, since $(u_\infty,v_\infty)\in \KK$, we also have
\begin{equation}
  \label{eq:227}
  \left| \int_{\RR^d}(1-\theta_n)\,u_\infty\,v_\infty\dd x\right|  \le  \frac{C}{n^{2/d}}\;.
\end{equation}
Thanks to~\eqref{b19}, \eqref{eq:226}, and~\eqref{eq:227}, we have for all $n\ge 1$
\begin{align*}
\left| \int_{\RR^d}(u_k\,v_k-u_\infty \,v_\infty)\dd x\right| &\le  \left| \int_{\RR^d}\theta_n\,(v_k-v_\infty)\,u_k\dd x\right| +  \left| \int_{\RR^d}\theta_n\,(u_k-u_\infty )\,v_\infty\dd x\right|\\
&\qquad +  \left| \int_{\RR^d}(1-\theta_n)\,u_k\,v_k\dd x\right|  +  \left| \int_{\RR^d}(1-\theta_n)\,u_\infty \,v_\infty\dd x\right|\\
&\le C\ \|\theta_n\,(v_k-v_\infty)\|_{m/(m-1)}+\left| \int_{\RR^d}\theta_n\,(u_k-u_\infty )\,v_\infty\dd x\right|\\
&\qquad+Cn^{-2/d}\;.
\end{align*}
Since $2<m/(m-1)<2d/(d-2)$ and $\theta_n v_\infty\in \LL^{m/(m-1)}(\RR^d)$, the first two terms of the right-hand side of the above inequality converge to zero as $k \to\infty$ by~\eqref{eq:222} and~\eqref{eq:225}. Thus, for all $n \ge 1$,
\begin{equation*}
  \limsup_{k\to \infty}  \left| \int_{\RR^d}(u_k\,v_k-u_\infty \,v_\infty)\dd x\right| \le Cn^{-2/d}\;.
\end{equation*}
Letting $n \to \infty$ gives
\begin{equation}
  \label{eq:228}
 \lim_{k\to \infty} \int_{\RR^d} u_k(x)\,v_k(x)\dd x=\int_{\RR^d}u_\infty(x)\, v_\infty(x)\dd x\;. 
\end{equation}
Combining the inequalities~\eqref{eq:224}, \eqref{eq:228} obtained previously, and the minimising property of $(u_k,v_k)$, we end up with $\FF_h[u_\infty,v_\infty] \le \omega$. Since $(u_\infty,v_\infty)$ belongs to $\KK$, we actually have $\FF_h[u_\infty,v_\infty] = \omega$ by~\eqref{eq:217} and $(u_\infty,v_\infty)$ is thus a minimiser of $\FF_h$ in $\KK$. 

Finally, given a minimiser $(u,v)$ of $\FF_h$ in $\KK$, the inequality~\eqref{eq:214} is a straightforward consequence of the minimising property of $(u,v)$.
\end{proof}

We next improve the regularity of minimisers of $\FF_h$ in $\KK$ by adapting an argument developed in~\cite{MMS09}. The key argument in~\cite{MMS09} is the existence of an additional Liapunov functional for the problem under study which is associated to another ``simpler'' gradient flow, the solutions to this gradient flow being then used as test functions. Though there does not seem to be such a structure hidden in the parabolic-parabolic Keller-Segel system~\eqref{resPKS}, we nevertheless find a simple evolution system, the solutions of which we use as test functions and thereby obtain additional regularity for the minimisers.
\begin{proposition}[Further regularity of the minimisers]\label{pr:b5}
Let $\chi \in (0,\chi_c)$, $(u_0,v_0) \in \KK$, and $h \in (0,1)$. Consider $(u,v) \in \KK$ be a minimiser of $\FF_h$ in $\KK$. Then, $u^{m/2}\in \HH^1(\RR^d)$, $\Delta v - \alpha v + u \in \LL^2(\RR^d)$ and there exists $C_2>0$ depending only on $d$, $\chi$, $\alpha$, and $\tau$ such that 
\begin{equation}
\frac{4}{m\chi}  \|\nabla(u^{m/2})\|_2^2 + \|\Delta v-\alpha v + u\|^2_2  \le 2 A_h(0) + C_2\left(\EE_\alpha[u_0,v_0] + \EE_\alpha[u_0,v_0]^{1/(m-1)}\right)  \label{eq:215} 
\end{equation}
where
\begin{equation}
A_h(0) := \frac{\mathcal{H}[u_0] - \mathcal{H}[u]}{h \chi} + \frac{\tau}{h} \,\left(\|\nabla v_0\|^2_2+\alpha\,\|v_0\|^2_2  - \|\nabla v\|^2_2-\alpha\,\|v\|_2^2\right) \label{b41b}
\end{equation}
and $\mathcal{H}$ is Boltzmann's entropy defined by
\begin{equation}
\mathcal{H}[w] := \int_{\RR^d} w(x)\,\ln{(w(x))}\, \dd x\,. \label{b33}
\end{equation}
\end{proposition}

Recall that, if $w\in (\PP_2\cap\LL^m)(\RR^d)$, then $w\, \ln{w}\in \LL^1(\RR^d)$ and there is $C_3>0$ depending only on $d$ such that
\begin{eqnarray}
\int_{\RR^d} w\, |\ln{w}| \dd x & \le & C_3 (1+\|w\|^m_m) + \int_{\RR^d} w(x)\ (1+|x|^2) \dd x\,, \nonumber\\
\mathcal{H}[w] & \ge & -C_3 - \int_{\RR^d} w(x)\ (1+|x|^2) \dd x \,, \label{spirou}
\end{eqnarray}
see, e.g.,~\cite[p.~329]{DPL89}. Therefore, $\mathcal{H}[u_0]$ and $\mathcal{H}[u]$ are well-defined and thus $A_h(0)$ as well. 

\begin{proof}[Proof of Proposition~\ref{pr:b5}]
Let $(u,v)$ be a minimiser of $\FF_h$ in $\KK$. We introduce the solutions $U$ and $V$ to the initial value problems
\begin{equation}
  \label{eq:229}
  \partial_t U - \Delta U=0\quad\mbox{in $(0,\infty) \times \RR^d$,}\qquad U(0)=u\quad\mbox{in $\RR^d$,}
\end{equation}
and
\begin{equation}
  \label{eq:230}
  \partial_t V - \Delta V + \alpha V=0\quad\mbox{in $(0,\infty) \times \RR^d$,}\qquad V(0)=v\quad\mbox{in $\RR^d$.}
\end{equation}
Classical results ensure that $(U(t),V(t))$ belongs to $\KK$ for all $t\ge 0$ and therefore
\begin{equation}
  \label{eq:231}
  \FF_h[u,v] \le \FF_h[U(t),V(t)]\;, \qquad t\ge 0\;.
\end{equation}
Unlike in \cite{MMS09}, the free energy $\EE_\alpha$ is not a Liapunov functional for $(U,V)$. Nevertheless, and this is actually the cornerstone of the proof, the time derivative of $\EE_\alpha$ along the flow of~\eqref{eq:229}--\eqref{eq:230} is the sum of a negative term and a remainder term which can be controlled. The claimed additional regularity on $(u,v)$ results from the computation of the time evolution of $\FF_h[U,V]$ and its proof requires several steps: we first use~\eqref{eq:229}, \eqref{eq:230}, and their gradient flow structures to compute the time evolution of $\FF_h[U,V]$. The second step is devoted to control a remainder term arising in this computation. Once this is done, weak convergence arguments and the lower bound on $\EE_\alpha$, see Lemma~\ref{lem:213} and Lemma~\ref{lem:213b}, are used in the last step to obtain~\eqref{eq:215}.

\noindent{\bf Step 1.}\\
$\bullet$ It follows from~\eqref{eq:229}, \eqref{eq:230}, and integration by parts that
\begin{align*}
  \frac{\dd}{\dd t}\EE_\alpha[U,V]&= -\frac{m}{\chi}\int_{\RR^d} U^{m-2}\,|\nabla U|^2\dd x - \int_{\RR^d} U\ \left( \Delta V - \alpha V + \alpha V \right)\dd x\notag\\
&\qquad - \int_{\RR^d} U\ \left( \Delta V - \alpha V \right) \dd x - \int_{\RR^d} (\Delta V -\alpha V)^2\dd x\notag\\
&= -\DD + \R\,,
\end{align*}
where
\begin{equation*}
  \DD(t):=\frac{4}{m\chi}\|\nabla \left(U^{m/2}(t)\right)\|^2_2 + \|(\Delta V -\alpha V+U)(t)\|^2_2\,, \quad t>0\,,
\end{equation*}
and
\begin{equation*}
  \R(t):= \|U(t)\|_2^2 - \alpha \int_{\RR^d} (UV)(t,x) \dd x\,, \quad t>0 \,.
\end{equation*}
After integration we obtain
\begin{equation}
  \EE_\alpha[U(t),V(t)] - \EE_\alpha[u,v] \le - \int_0^t\DD(s)\dd s+\int_0^t\R(s)\dd s\,, \quad t>0\;. \label{8p8}
\end{equation}
$\bullet$ We next recall that the linear heat equation \eqref{eq:229} can be interpreted as the gradient flow of the functional $\mathcal{H}$ defined in \eqref{b33} for the Kantorovich-Wasserstein distance $\WW_2$ in $\PP_ 2(\RR^d)$, see~\cite{JKO98,Ot01}. Therefore, owing to the differentiability properties of the Kantorovich-Wasserstein distance (see, e.g., \cite[Corollary~10.2.7]{AGS08} and \cite[Theorem~8.13]{Vi03}), it follows from \cite[Theorem~11.1.4]{AGS08} that 
\begin{equation*}
  \frac12 \frac{\dd}{\dd t} \WW_2^2(U(t),u_0) \le \mathcal{H}[u_0] - \mathcal{H}[U(t)] \,,\quad t> 0\;. 
\end{equation*}
Integrating with respect to time we obtain
\begin{equation*}
   \frac12  \left[\WW_2^2(U(t),u_0) -  \WW_2^2(u,u_0)\right] \le \int_0^t \left( \mathcal{H}[u_0] - \mathcal{H}[U(s)] \right)\dd s\;.
\end{equation*}
By the monotonicity of $s \mapsto \mathcal{H}[U(s)]$ we deduce that, for all $t>0$,
\begin{equation}\label{eq:233}
   \frac12  \left[\WW_2^2(U(t),u_0) -  \WW_2^2(u,u_0)\right] \le t \left( \mathcal{H}[u_0] - \mathcal{H}[U(t)] \right)\;.
\end{equation}

\noindent $\bullet$ Furthermore, it readily follows from~\eqref{eq:230} and Young's inequality that 
\begin{align*}
   \frac12 \frac{\dd}{\dd t}  \|V-v_0\|^2_2 &=- \int_{\RR^d} \left[ \nabla(V-v_0) \cdot \nabla V + \alpha\, V\ (V-v_0) \right] \dd x\\
&\le -\|\nabla V\|^2_2+\frac12 \left(\|\nabla V\|^2_2 + \|\nabla v_0\|^2_2 \right)- \alpha \,\|V\|^2_2 \\
&\qquad +\frac\alpha2 \,\left(\|V\|^2_2 + \|v_0\|^2_2 \right)\\
&\le -\frac12 \left(\|\nabla V\|^2_2 + \alpha\, \|V\|^2_2 \right) + \frac12 \left(\|\nabla v_0\|^2_2 + \alpha\, \|v_0\|^2_2 \right)\;.
\end{align*}
Integrating with respect to time and using the monotonicity of $s \mapsto \|\nabla V(s)\|^2_2$ and $s \mapsto \| V(s)\|^2_2$ we end up with
\begin{equation}\label{eq:234}
 \|V(t)-v_0\|^2_2 - \|v-v_0\|^2_2\le t \left( \|\nabla v_0\|^2_2 + \alpha \,\|v_0\|^2_2  -\|\nabla V(t)\|^2_2 - \alpha \,\|V(t)\|^2_2 \right)
\end{equation} for all $t>0$.

$\bullet$ Combining the above estimates~\eqref{eq:231}, \eqref{8p8}, \eqref{eq:233} and~\eqref{eq:234} gives, for $t>0$,
\begin{align*}
  0 &\le \FF_h[U(t),V(t)]- \FF_h[u,v]\\
&\le \frac{t}{h\chi}\left( \mathcal{H}[u_0] - \mathcal{H}[U(t)] \right) - \int_0^t\DD(s)\dd s+\int_0^t\R(s)\dd s\\
&\qquad + \frac{\tau t}{2h} \left( \|\nabla v_0\|^2_2 + \alpha\, \|v_0\|^2_2  -\|\nabla V(t)\|^2_2 - \alpha \,\|V(t)\|^2_2 \right)\;,
\end{align*}
which also reads
\begin{equation}
  \frac1t\int_0^t\DD(s)\dd s\le A_h(t) +\frac1t\int_0^t\R(s)\dd s\,, \quad t>0\;,  \label{eq:235}
\end{equation}
with
$$
A_h(t) := \frac{\mathcal{H}[u_0] - \mathcal{H}[U(t)]}{h\chi} + \frac{\tau}{2h} \left( \|\nabla v_0\|^2_2 + \alpha\, \|v_0\|^2_2  -\|\nabla V(t)\|^2_2 - \alpha \,\|V(t)\|^2_2 \right)\,.
$$
Owing to the continuity and regularity properties of $U$ and $V$, the function $A_h$ has a limit as $t\to 0$ and
\begin{equation}
\lim_{t\to 0} A_h(t) = A_h(0) \;, \label{b41}
\end{equation}
the constant $A_h(0)$ being defined in \eqref{b41b}.

\medskip

\noindent{\bf Step~2.} We now estimate $\R$. Since $2\in (m,2/(m-1))$, the H\"older and Sobolev inequalities ensure that, given $w\in\LL^m(\RR^d)$ such that $|w|^{m/2}\in \HH^1(\RR^d)$, the function $w$ belongs to $\LL^2(\RR^d)$ and
\begin{equation}
\|w\|_2^2 \le \|w\|_m\ \|w\|_{m/(m-1)} \le C\ \|w\|_m\ \left\| \nabla (|w|^{m/2}) \right\|_{2}^{2/m}\,. \label{b42}
\end{equation}
Combining the above estimate with Young's inequality gives
\begin{equation*}
\|U\|_2^2 \le \frac{2}{m\chi} \left\| \nabla (U^{m/2}) \right\|_2^2 + C\ \|U\|_m^{m/(m-1)}\,,
\end{equation*}
and thus
\begin{equation}
\|U\|_2^2  \le  \frac{\DD}{2} + C\ \|U\|_m^{m/(m-1)}\,. \label{b43}
\end{equation}
Moreover, we infer from~\eqref{220}  and Young's inequality that
\begin{equation}
\alpha \|UV\|_1 \le \alpha C \|U\|_m^{m/2}\ \|U\|_1^{1/d}\ \|\nabla V\|_2 \le \frac{\|\nabla V\|_2^2}{2}  + \alpha^2 C \|U\|_m^m \;.\label{b44}
\end{equation}
We now infer from~\eqref{b43}, \eqref{b44}, and the time monotonicity of $s \mapsto \|\nabla V(s)\|^2_2$ and $s \mapsto \|U(s)\|^m_m$ that
$$
\frac{1}{t} \int_0^t \R(s) \dd s \le \frac{1}{2t} \int_0^t \DD(s) \dd s + \frac{1}{2}\ \|\nabla v\|_2^2 + C \left( \|u\|_m^m + \|u\|_m^{m/(m-1)} \right)\,.
$$
By~\eqref{eq:235}, Lemma~\ref{lem:213} and Lemma~\ref{lem:213b} we finally obtain
\begin{eqnarray}
\frac{1}{t} \int_0^t \DD(s) \dd s & \le & 2 A_h(t) + \|\nabla v\|_2^2 + C \left( \|u\|_m^m + \|u\|_m^{m/(m-1)} \right) \nonumber\\
& \le & 2 A_h(t) + C \left( \EE_\alpha[u,v] + \EE_\alpha[u,v]^{1/(m-1)} \right) \,. \label{b45}
\end{eqnarray}

\noindent{\bf Step~3.} We are now left with passing to the limit as $t\to 0$ in the left-hand side of~\eqref{b45}. To this end, define first 
$$
\DD_1(t,x) := \frac{1}{t}\int_0^t U^{m/2}(s,x) \dd s\,, \quad (t,x)\in [0,1]\times\RR^d\,.
$$
Since $m\in (1,2)$ and $U\in \C([0,1];\LL^m(\RR^d))$, the function $\DD_1$ belongs to $\C([0,1];\LL^2(\RR^d))$ and $\DD_1(t)$ converges to $u^{m/2}$ in $\LL^2(\RR^d)$ as $t\to 0$. Combining this property with~\eqref{b41} and~\eqref{b45}, we realise that $(\nabla \DD_1(t))_{t\in (0,1)}$ is bounded in $\LL^2(\RR^d;\RR^d)$ and converges towards $\nabla \left(u^{m/2}\right)$ in $\HH^{-1}(\RR^d;\RR^d)$. Therefore, $\nabla \left(u^{m/2}\right)$ belongs to $\LL^2(\RR^d;\RR^d)$ and
\begin{equation}
\left\| \nabla (u^{m/2}) \right\|_2^2 \le \liminf_{t\to 0} \left\| \nabla \DD_1(t) \right\|_2^2\,. \label{b47}
\end{equation}
Similarly, defining 
$$
\DD_2(t,x) := \frac{1}{t}\ \left( \Delta \int_0^t V(s,x) \dd s - \alpha \int_0^t V(s,x) \dd s + \int_0^t U(s,x) \dd s \right)
$$
for $(t,x)\in [0,1]\times\RR^d$, we infer from~\eqref{b41} and~\eqref{b45} that $(\DD_2(t))_{t\in (0,1)}$ is bounded in $\LL^2(\RR^d)$ while the continuity properties of $U$ and $V$ with respect to time guarantee that $\DD_2(t)$ converges towards $\Delta v - \alpha v + u$ in $\HH^{-1}(\RR^d)$ as $t\to 0$. Consequently, $\Delta v - \alpha v + u\in\LL^2(\RR^d)$ and
\begin{equation}
\left\| \Delta v - \alpha v + u \right\|_2^2 \le \liminf_{t\to 0} \left\| \DD_2(t) \right\|_2^2\,. \label{b48}
\end{equation}
Thanks to~\eqref{b41}, \eqref{b47}, and~\eqref{b48}, we may let $t\to 0$ in~\eqref{b45} and obtain the stated result.
\end{proof}
\begin{corollary}\label{cor:b6}
Let $\chi \in (0,\chi_c)$, $(u_0,v_0) \in \KK$, $h\in (0,1)$, and consider a minimiser $(u,v)$ of $\FF_h$ in $\KK$. Then $u\in\LL^2(\RR^d)$.
\end{corollary}
\begin{proof}
Since $u$ belongs to $\LL^m(\RR^d)$ by the definition of $\KK$, Corollary~\ref{cor:b6} follows at once from Proposition~\ref{pr:b5} and \eqref{b42}.
\end{proof}
\subsection{The Euler-Lagrange equation}
\begin{lemma}[Euler-Lagrange equation]\label{lem:b7}
Consider $\chi\in (0,\chi_c)$, $(u_0,v_0)\in\KK$, and $h\in (0,1)$. If $(u,v) \in \KK$ is a minimiser of $\FF_h$ in $\KK$, then 
\begin{equation}
\left| \int_{\RR^d} \left[ \xi (u-u_0) + h\ \nabla\xi\cdot \left( \nabla \left(u^m\right) - \chi\ u\ \nabla v \right) \right] \dd x \right| \le \|\xi\|_{\rmW^{2,\infty}} \frac{\WW_2^2(u,u_0)}{2}\ \label{bb40}
\end{equation}
for any $\xi$ in $\C^\infty_0(\RR^d)$ and
\begin{equation}
\tau\ \frac{v-v_0}{h} -\Delta v +\alpha \,v - u = 0 \;\;\text{ a.e. in }\;\; \RR^d\;. \label{bb41}
\end{equation}
In addition, $\nabla \left(u^m\right) - \chi\ u\ \nabla v\in \LL^{2m/(2m-1)}(\RR^d)$ and satisfies
\begin{equation}
h\ \left\| \nabla \left(u^m\right) - \chi\ u\ \nabla v\right\|_{2m/(2m-1)} \le C_4\ \WW_2(u,u_0)\ \left\|\nabla \left(u^{m/2}\right) \right\|_2^{1/m} \label{bb42}
\end{equation}
for some positive constant $C_4$ depending only on $d$, $\alpha$, $\chi$ and $\tau$.
\end{lemma}
\begin{proof}
Pick two smooth test functions $\zeta\in\C^\infty_0(\RR^d;\RR^d)$ and $w\in\C^\infty_0(\RR^d)$ and define $T_\delta:=\id + \delta\,\zeta$ and
\begin{equation*}
 u_\delta= T_\delta \# u\;, \quad v_\delta := v + \delta\, w 
\end{equation*}
for $\delta\in (0,1)$, where $\id$ is the identity function of $\RR^d$ and $T\#\mu$ denotes the image measure or push-forward measure of the measure $\mu$ by the map $T$. Notice that there is $\delta_\zeta$ small enough such that $T_\delta$ is a $\C^\infty$-diffeomorphism from $\RR^d$ onto $\RR^d$ for all $\delta\in (0,\delta_\zeta)$ and ${\rm Det}\, (\nabla T_\delta) = {\rm Det}\, ({\rm I} + \delta\,\nabla \zeta) >0$. 

\noindent $\bullet$ By a standard computation, see~\cite[Theorem~5.30]{Vi03} for instance, 
\begin{equation}
\label{bb46}
\lim_{\delta\to 0}  \frac{\|u_\delta\|^m_m-\|u\|^m_m}{(m-1)\delta} = - \int_{\RR^d}\tr(\nabla \zeta(x))\,u^m(x)\dd x\;.
\end{equation}

\noindent $\bullet$ It is also standard, see~\cite[Theorem~8.13]{Vi03} that
\begin{equation}
 \lim_{\delta \to 0} \frac{\WW^2_2(u_\delta ,u_0)-\WW^2_2(u ,u_0)}{2\delta}= - \int_{\RR^d} (x-\nabla \varphi(x))\cdot \zeta\!\circ\!\nabla\varphi(x)\,u_0(x)\dd x\,, \label{bb49}
\end{equation}
where $\nabla\varphi$ is the optimal map pushing $u_0$ onto $u$, that is, $u=\nabla\varphi\# u_0$ and
\begin{eqnarray*}
\WW^2_2(u,u_0) & = & \int_{\RR^d} |x-\nabla\varphi(x)|^2\ u_0(x) \dd x \\
& = & \inf{ \left\{ \int_{\RR^d} |x-T(x)|^2\ u_0(x) \dd x\ :\ T\#u_0=u \right\} }\;.
\end{eqnarray*}

\noindent $\bullet$ Another classical computation ensures that
\begin{equation}
\lim_{\delta\to 0} \frac{1}{2\delta}\ \left[ \|\nabla v_\delta\|^2_2 + \alpha\, \|v_\delta\|^2_2-\|\nabla v\|^2_2 - \alpha\, \|v\|^2_2 \right] = \int_{\RR^d} \left( \nabla v\cdot\nabla w + \alpha v\, w \right) \dd x\;. \label{bb47}
\end{equation}

\noindent $\bullet$ Finally, by the definition of the push-forward measure, 
\begin{equation*}
\int_{\RR^d} (u\,v-u_\delta\,v_\delta)(x) \dd x = \int_{\RR^d} u(x) \left[v(x)-v(T_\delta(x))-\delta\, w(T_\delta(x)) \right]\dd x\;.
\end{equation*}
Since $u$ is bounded in $\LL^2(\RR^d)$ by Corollary~\ref{cor:b6} and 
$$ 
\frac{v\!\circ\!T_\delta -v}{\delta} \rightharpoonup \zeta \cdot \nabla v\quad \mbox{in $\LL^2(\RR^d)$,} \qquad w\!\circ\!T_\delta  \rightarrow w \quad \mbox{in $\LL^2(\RR^d),$}
$$
we conclude that
\begin{equation}
\label{bb48}
\lim_{\delta \to 0} \frac1\delta\int_{\RR^d} [u\,v-u_\delta\,v_\delta](x) \dd x =  - \int_{\RR^d} \left( u\, \zeta \cdot \nabla v + u\, w \right) \dd x\;.
\end{equation}

\noindent $\bullet$ Since $(u_\delta,v_\delta)$ belongs to $\KK$, we infer from the above estimates~\eqref{bb46}, \eqref{bb49}, \eqref{bb47}, and~\eqref{bb48} that 
\begin{align*}
0 &\le \lim_{\delta\to 0} \frac1\delta \left( \FF_\tau[u_\delta,v_\delta]-\FF_\tau[u,v]\right)\\
&= -\frac1{h \chi} \int_{\RR^d}(x-\nabla \varphi(x))\cdot \zeta\!\circ\!\nabla\varphi(x)\,u_0(x)\dd x + \frac{\tau}{h}\int_{\RR^d} w(x)\,(v(x)-v_0(x)) \dd x\\
&\qquad-\frac1\chi \int_{\RR^d}\tr (\nabla \zeta(x))\,u^m(x) \dd x-\int_{\RR^d} u(x)\, \zeta(x)\cdot \nabla v(x)\dd x \\
&\qquad -  \int_{\RR^d}u(x)\, w(x) \dd x+ \int_{\RR^d} [-\Delta v(x) +\alpha \,v(x)] \, w(x)\dd x\;.
\end{align*}
The above inequality being valid for arbitrary $(\zeta,w)\in\C_0^\infty(\RR^d;\RR^d)\times\C_0^\infty(\RR^d)$, it is also valid for $(-\zeta,-w)$ so that we end up with
\begin{eqnarray}
& & \frac{1}{\chi} \int_{\RR^d} \zeta\cdot \left( \nabla \left(u^m\right) - \chi\ u\ \nabla v \right) \dd x + \int_{\RR^d} \left( \tau\ \frac{v-v_0}{h} - \Delta v + \alpha v - u \right)\ w \dd x \label{bb50}\\
& = & \frac1{h \chi} \int_{\RR^d}(x-\nabla \varphi(x))\cdot \zeta\!\circ\!\nabla\varphi(x)\,u_0(x)\dd x\,. \nonumber
\end{eqnarray}
Observe that, since $\nabla \left(u^m\right) = 2\ u^{m/2}\ \nabla \left(u^{m/2}\right)$ and $u\in\LL^m(\RR^d)$ with $\nabla \left(u^{m/2}\right)\in\LL^2(\RR^d)$ by the definition of $\KK$ and Proposition~\ref{pr:b5}, $\nabla u^m$ belongs to $\LL^1(\RR^d)$ and the first term on the left-hand side of \eqref{bb50} is meaningful.

Now, taking $\zeta=0$ in \eqref{bb50} and using a density argument, the regularity of the minimisers, see Proposition~\ref{pr:b5} and Corollary~\ref{cor:b6}, readily gives \eqref{bb41}. We next take $w=0$ in \eqref{bb50} to obtain
\begin{equation}
\int_{\RR^d} \zeta\cdot \left( \nabla \left(u^m \right)- \chi\ u\ \nabla v \right) \dd x = \frac{1}{h} \int_{\RR^d}(x-\nabla \varphi(x))\cdot \zeta\!\circ\!\nabla\varphi(x)\,u_0(x)\dd x \label{bb51}
\end{equation}
for all $\zeta\in\C_0^\infty(\RR^d;\RR^d)$. 

On the one hand, the Cauchy-Schwarz inequality and the properties of $\nabla\varphi$ ensure that
$$
\left| \int_{\RR^d}(x-\nabla \varphi(x))\cdot \zeta\!\circ\!\nabla\varphi(x)\,u_0(x)\dd x \right| \le \WW_2(u,u_0)\ \left( \int_{\RR^d} |\zeta(x)|^2\ u(x) \dd x \right)^{1/2}
$$
and we infer from the H\"older and Sobolev inequalities that
\begin{eqnarray*}
\left| \int_{\RR^d}(x-\nabla \varphi(x))\cdot \zeta\!\circ\!\nabla\varphi(x)\,u_0(x)\dd x \right| & \le & \WW_2(u,u_0)\ \|\zeta\|_{2m}\ \left\| u^{m/2} \right\|_{2/(m-1)}^{1/m} \\
& \le & C\ \WW_2(u,u_0)\ \|\zeta\|_{2m}\ \left\| \nabla \left(u^{m/2}\right) \right\|_2^{1/m}\,.
\end{eqnarray*}
Recalling \eqref{bb51} we arrive at
$$
\left| \int_{\RR^d} \zeta\cdot \left( \nabla \left(u^m\right)  - \chi\ u\ \nabla v \right) \dd x \right| \le C\ \WW_2(u,u_0)\ \|\zeta\|_{2m}\ \left\| \nabla \left(u^{m/2} \right) \right\|_2^{1/m}
$$
for all $\zeta\in\C_0^\infty(\RR^d;\RR^d)$, whence $\nabla \left(u^m\right)  - \chi u\nabla v\in\LL^{2m/(2m-1)}(\RR^d)$ and \eqref{bb42} by a duality argument.

On the other hand, consider $\xi\in\C_0^\infty(\RR^d)$. Using Taylor's expansion, we have
\begin{equation*}
\left| \xi(x) - \xi (\nabla \varphi(x)) - (\nabla\xi\circ\nabla\varphi)(x)\cdot (x-\varphi(x)) \right| \le \|D^2\xi\|_\infty\ \frac{|x-\nabla\varphi(x)|^2}{2}
\end{equation*}
for $x\in\RR^d$. Multiplying the above inequality by $u_0$, integrating over $\RR^d$, and using the properties of $\nabla\varphi$ give
$$
\left| \int_{\RR^d} \left[ \xi\, u_0 - \xi\, u - (\nabla\xi\circ\nabla\varphi)\cdot (\id-\varphi)\, u_0 \right] \dd x \right| \le \|D^2\xi\|_\infty\ \WW_2^2(u,u_0)\,.
$$
Combining the above inequality with \eqref{bb51} (with $\zeta=\nabla\xi$) leads us to \eqref{bb40}.
\end{proof}

\section{Convergence}\label{sec:conv}

Let $\chi\in (0,\chi_c)$, $(u_0,v_0)\in\KK$, and $h\in (0,1)$. We define a sequence $(u_{h,n},v_{h,n})_{n\ge 0}$ in $\KK$ as follows:
\begin{equation}
(u_{h,0},v_{h,0}) := (u_0,v_0)\,, \label{cc1}
\end{equation}
and, for each $n\ge 0$, 
\begin{eqnarray*}
& & (u_{h,n+1},v_{h,n+1}) \;\;\text{ is a minimiser in $\KK$ of the functional} \\
& & \mathcal{F}_{h,n}[u,v] := \frac{1}{2h}\ \left[ \frac{\WW_2^2(u,u_{h,n})}{\chi} + \tau\ \|v-v_{h,n}\|_2^2 \right] + \EE_\alpha[u,v]\,, \quad (u,v)\in\KK\,. \nonumber
\end{eqnarray*}
Recall that the existence of $(u_{h,n+1},v_{h,n+1})$ is guaranteed by Proposition~\ref{pr:b4} since $\chi\in (0,\chi_c)$. We next define a piecewise constant time dependent pair of functions $(u_h,v_h)$ by
\begin{equation*}
(u_h,v_h)(t) := (u_{h,n},v_{h,n})\,, \qquad t\in [nh,(n+1)h)\,, \qquad n\ge 0\,. 
\end{equation*}

\subsection{Compactness}\label{sec:comp}

By the analysis performed in Section~\ref{sec:jko}, $(u_h(t),v_h(t))$ belongs to $\KK$ for all $t\ge 0$ and  $(u_h,v_h)$ is endowed with several interesting properties which we collect in the next lemma.
\begin{lemma}[Uniform estimates]\label{lem:cc1} There is a positive constant $C_5$ depending only on $d$, $\chi$, $\tau$, $\alpha$, $u_0$, and $v_0$ such that, for $0\le s \le t$, 
\begin{equation}
\EE_\alpha[u_h(t),v_h(t)]  \le  \EE_\alpha[u_h(s),v_h(s)]\,,\label{cc4} \tag{i}
\end{equation}
\begin{equation}
\left( \sum_{n=0}^\infty \frac{\WW_2^2(u_{h,n+1},u_{h,n})}{\chi} + \tau\ \|v_{h,n+1} - v_{h,n}\|_2^2 \right) \le 2\EE_\alpha[u_0,v_0]\ h\,,  \label{cc5}\tag{ii}
\end{equation}
\begin{equation}
\|u_h(t)\|_m + \|\nabla v_h(t)\|_2  \le  C_5\,, \label{cc7}\tag{iii}
\end{equation}
\begin{equation}
\|v_h(t)\|_2^2 + \int_{\RR^d} |x|^2\ u_h(t,x)\dd x  \le  C_5\ (1+t)\,,\label{cc6} \tag{iv}
\end{equation}
and, for $t\ge h$,
\begin{equation}
\int_h^t \left( \left\| \nabla \left( u_h^{m/2} \right)(s) \right\|_2^2 + \| (\Delta v_h - \alpha v_h + u_h)(s)\right) \|_2^2 \dd s  \le  C_5\ (1+t)\,, \label{cc8}\tag{v}
\end{equation}
\begin{equation}
\int_h^t \left\| (\nabla (u_h^m) - \chi u_h \nabla v_h)(s) \right\|_{2m/(2m-1)}^{2m/(m+1)} \dd s  \le  C_5\ (1+t)\,. \label{cc9}\tag{vi}
\end{equation}
\end{lemma}

\begin{proof}
$\bullet$ {\bf Energy and moment estimates:} Let $n\ge 0$. It follows from the properties of the minimisers of $\FF_{h,n}$ in $\KK$, see Proposition~\ref{pr:b4}, that
\begin{equation}
\frac{1}{2h}\ \left[ \frac{\WW_2^2(u_{h,n+1},u_{h,n})}{\chi} + \tau\ \|v_{h,n+1}-v_{h,n}\|_2^2 \right] + \EE_\alpha[u_{h,n+1},v_{h,n+1}] \le \EE_\alpha[u_{h,n},v_{h,n}]\,, \label{cc10}
\end{equation}

Consider $s\ge 0$, $t\in (s,\infty)$, and set $N:=[t/h]$ and $\nu:=[s/h]$. If $N\ge 1+\nu$, we sum up \eqref{cc10} from $n=\nu$ to $n=N-1$ to obtain
$$
\frac{1}{2h}\ \sum_{n=\nu}^{N-1} \left[ \frac{\WW_2^2(u_{h,n+1},u_{h,n})}{\chi} + \tau\ \|v_{h,n+1}-v_{h,n}\|_2^2 \right] + \EE_\alpha[u_{h,N},v_{h,N}] \le \EE_\alpha[u_{h,\nu},v_{h,\nu}]\,.
$$
Since both terms of the left-hand side of the above inequality are non-negative by Lemma~\ref{lem:213} and Lemma~\ref{lem:213b}, we deduce \eqref{cc4} as 
\begin{equation*}
\EE_\alpha[u_h(t),v_h(t)] = \EE_\alpha[u_{h,N},v_{h,N}]  \le  \EE_\alpha[u_{h,\nu},v_{h,\nu}] = \EE_\alpha[u_h(s),v_h(s)]\,, 
\end{equation*}
and \eqref{cc5} by taking $s=0$ ($\nu=0$) and $t\to\infty$ in
\begin{equation*}
\sum_{n=\nu}^{N-1} \left[ \frac{\WW_2^2(u_{h,n+1},u_{h,n})}{\chi} + \tau\ \|v_{h,n+1}-v_{h,n}\|_2^2 \right]  \le  2h\, \EE_\alpha[u_{h,\nu},v_{h,\nu}]\,.
\end{equation*}
The estimate \eqref{cc7} directly follows from the lower bound on $\EE_\alpha$, see Lemma~\ref{lem:213} and Lemma~\ref{lem:213b}, while we deduce \eqref{cc6} from \eqref{cc5} thanks to the inequalities
\begin{align*}
\WW_2(\delta_0,u_{h,N}) & \le \WW_2(\delta_0,u_0) + \sum_{n=0}^{N-1} \WW_2(u_{h,n+1},u_{h,n}) \\
& \le \WW_2(\delta_0,u_0) + \sqrt{N}\ \left( \sum_{n=0}^{N-1} \WW_2^2(u_{h,n+1},u_{h,n}) \right)^{1/2}
\end{align*}
and
\begin{align*}
\|v_{h,N}\|_2 & \le \| v_0\|_2 + \sum_{n=0}^{N-1} \|v_{h,n+1} - v_{h,n}\|_2 \\
& \le \| v_0\|_2  + \sqrt{N}\ \left( \sum_{n=0}^{N-1} \|v_{h,n+1} - v_{h,n}\|_2^2 \right)^{1/2}\,.
\end{align*}

\noindent$\bullet$ {\bf Additional estimates:} Assume now that $t\ge h$. By the properties of the minimisers of $\FF_{h,n}$ in $\KK$, see Proposition~\ref{pr:b5},
\begin{align*}
\frac{4h}{m\chi}  &\|\nabla(u_{h,n+1}^{m/2})\|_2^2  +  h\ \|\Delta v_{h,n+1}-\alpha v_{h,n+1} + u_{h,n+1}\|^2_2 \\
& \le  \frac{2}{\chi}\ \left( \mathcal{H}[u_{h,n}] - \mathcal{H}[u_{h,n+1}] \right)+  C_2\, h\, \left(\EE_\alpha[u_{h,n},v_{h,n}] + \EE_\alpha[u_{h,n},v_{h,n}]^{1/(m-1)}\right) \nonumber \\
&\qquad +  \tau \,\left(\|\nabla v_{h,n}\|^2_2+\alpha\,\|v_{h,n}\|^2_2  - \|\nabla v_{h,n+1}\|^2_2-\alpha\,\|v_{h,n+1}\|^2 \right) \,. \nonumber
\end{align*}
Summing from $n=0$ to $N-1$, it follows from \eqref{spirou}, \eqref{cc1}, \eqref{cc4}, and \eqref{cc6} that 
\begin{align*}
 \frac{4h}{m\chi}&  \sum_{n=0}^{N-1} \|\nabla(u_{h,n+1}^{m/2})\|_2^2 + h\ \sum_{n=0}^{N-1} \|\Delta v_{h,n+1}-\alpha v_{h,n+1} + u_{h,n+1}\|^2_2 \\
& \le  \frac{2}{\chi}\ \left( \mathcal{H}[u_{h,0}] - \mathcal{H}[u_{h,N}] \right)+  C_2\, h\, \sum_{n=0}^{N-1} \left(\EE_\alpha[u_{h,n},v_{h,n}] + \EE_\alpha[u_{h,n},v_{h,n}]^{1/(m-1)}\right) \\
& \qquad+  \tau \,\left(\|\nabla v_{h,0}\|^2_2+\alpha\,\|v_{h,0}\|^2_2  - \|\nabla v_{h,N}\|^2_2-\alpha\,\|v_{h,N}\|_2^2 \right) \\
& \le  \frac{2}{\chi}\ \left( \mathcal{H}[u_{h,0}] + C_3 + \int_{\RR^d} |x|^2 u_{h,N}(x)\dd x \right) + \tau \,\left(\|\nabla v_{h,0}\|^2_2+\alpha\,\|v_{h,0}\|^2_2 \right)\\
& \qquad+  C_2\, h\, \sum_{n=0}^{N-1} \left(\EE_\alpha[u_{h,0},v_{h,0}] + \EE_\alpha[u_{h,0},v_{h,0}]^{1/(m-1)}\right) \\
&\le C\ (1+t)\,,
\end{align*}
whence \eqref{cc8}.

Finally, setting $\vartheta:= 2m/(m+1)\in (1,2)$, we infer from~\eqref{bb42}, \eqref{cc5}, \eqref{cc8}, and H\"older's inequality that 
\begin{align*}
  h\, \sum_{n=0}^{N-1}& \left\| \nabla \left(u_{h,n+1}^m\right) - \chi\, u_{h,n+1}\, \nabla v_{h,n+1} \right\|_{2m/(2m-1)}^\vartheta \\
& \le  C_4^{\vartheta}\ h^{1-\vartheta}\ \sum_{n=0}^{N-1} \WW_2^{\vartheta}(u_{h,n+1},u_{h,n})\ \left\| \nabla \left(u_{h,n+1}^{m/2}\right) \right\|_2^{\vartheta/m} \\
& \le C\ h^{1-\vartheta}\ \left( \sum_{n=0}^{N-1} \WW_2^2(u_{h,n+1},u_{h,n}) \right)^{\vartheta/2}\ \left( \sum_{n=0}^{N-1} \left\| \nabla \left(u_{h,n+1}^{m/2}\right) \right\|_2^2 \right)^{(2-\vartheta)/2}\\
& \le  C\ \left( \int_h^{(N+1)h} \left\| \nabla \left(u_h^{m/2}(s)\right) \right\|_2^2\dd s \right)^{(2-\vartheta)/2}\\
&  \le C(1+t)\,,
\end{align*}
which completes the proof.
\end{proof}

The following corollary guarantees the compactness of $(u_h)$, $(v_h)$, and $(\nabla v_h)$ with respect to the space variable $x$:
\begin{corollary}\label{cor:cc2} There is a positive constant $C_6$ depending only on $d$, $\chi$, $\tau$, $\alpha$, $u_0$, and $v_0$ such that, for $t\ge h$, 
\begin{eqnarray*}
\int_h^t \left( \|u_h(s)\|_2^{2m} + \|\nabla u_h(s)\|_m^2 \right) \dd s  &\le&  C_6\, (1+t)\,, \\
\int_h^t \| v_h(s)\|_{\HH^2}^2 \dd s  &\le&  C_6\ (1+t)\,. 
\end{eqnarray*}
\end{corollary}

\begin{proof}
We infer from the regularity of the minimisers of $\FF_{h,n}$ in $\KK$, see Proposition~\ref{pr:b5}, \eqref{b42}, and the estimates of Lemma~\ref{lem:cc1} that 
\begin{equation}
\int_h^t \|u_h(s)\|_2^{2m} \dd s \le C\ \int_h^t \|u_h(s)\|_m^m\ \left\|\nabla (u_h^{m/2})(s)\right\|_2^2 \dd s \le C\ (1+t)\,. \label{fantasio}
\end{equation}

Next, since $m\in (1,2)$, we have $\nabla u_h = 2 u_h^{(2-m)/2}\, \nabla (u_h^{m/2})/m$ and we infer from the estimates \eqref{cc7} and \eqref{cc8} of Lemma~\ref{lem:cc1} that 
$$
\int_h^t \| \nabla u_h(s) \|_m^2 \dd s \le \frac{4}{m^2}\ \int_h^t \left\|\nabla (u_h^{m/2})(s) \right\|_2^2\ \|u_h(s)\|_m^{2-m} \dd s \le C\ (1+t)\,.
$$
Finally, since $m>1$, we combine the estimates \eqref{cc7}--\eqref{cc8} of Lemma~\ref{lem:cc1} with \eqref{fantasio} to conclude.
\end{proof}

We now turn to the compactness with respect to time.

\begin{lemma}[Time equicontinuity]\label{lem:cc3} There is a positive constant $C_7$ depending only on $d$, $\chi$, $\tau$, $\alpha$, $u_0$, and $v_0$ such that, for $0\le s \le t$, 
\begin{eqnarray*}
\|v_h(t)-v_h(s)\|_2 & \le & C_7\ \left( \sqrt{t-s} + \sqrt{h} \right)\,,  \\
\|u_h(t)-u_h(s)\|_{\HH^{-(d+2)}} & \le & C_7\ (1+t)^{(m+1)/2m}\ (t-s+h)^{(m-1)/2m}\,.
\end{eqnarray*}
\end{lemma}

\begin{proof}
Consider $s\ge 0$, $t\ge s$ and set $N=[t/h]$ and $\nu=[s/h]$. Either $N=\nu$ and $(u_h,v_h)(t)=(u_h,v_h)(s)$ which readily implies the result. Or $N>\nu$ and, on the one hand, we infer from the estimate~\eqref{cc5} of Lemma~\ref{lem:cc1} that
\begin{align*}
\|v_h(t)-v_h(s)\|_2 & =  \|v_{h,N}-v_{h,\nu}\|_2 \le \sum_{n=\nu}^{N-1} \|v_{h,n+1}-v_{h,n}\|_2 \\
& \le  \sqrt{N-\nu}\ \left( \sum_{n=\nu}^{N-1} \|v_{h,n+1}-v_{h,n}\|_2^2 \right)^{1/2} \\
& \le  \sqrt{\frac{2 \,(N-\nu)\,h\, \EE_\alpha[u_0,v_0]}{\tau}} \le C\ \left( \sqrt{t-s} + \sqrt{h} \right)\,.
\end{align*}
On the other hand, it follows from the Euler-Lagrange equation~\eqref{bb40} that, for $\xi\in\C_0^\infty(\RR^d)$ and $\nu\le n \le N-1$, 
\begin{align*}
\left| \int_{\RR^d}  \xi (u_{h,n+1}-u_{h,n}) \dd x \right| & \le  h\ \int_{\RR^d} |\nabla\xi|\ \left| \nabla \left(u_{h,n+1}^m \right)- \chi\ u_{h,n+1}\ \nabla v_{h,n+1} \right| \dd x \\
& \qquad+ \|\xi\|_{\rmW^{2,\infty}}\  \frac{\WW_2^2(u_{h,n+1},u_{h,n})}{2}\,,
\end{align*}
whence, after summing up these inequalities from $n=\nu$ to $n=N-1$:
\begin{align*}
\left| \int_{\RR^d}  \xi (u_h(t)-u_h(s)) \dd x \right| & =  \left| \int_{\RR^d}  \xi (u_{h,N}-u_{h,\nu}) \dd x \right| \\
& \le  \sum_{n=\nu}^{N-1} \left| \int_{\RR^d}  \xi (u_{h,n+1}-u_{h,n}) \dd x \right| \\
& \le  h\sum_{n=\nu}^{N-1} \|\nabla\xi\|_{2m}\ \left\| \nabla \left(u_{h,n+1}^m \right)- \chi\ u_{h,n+1}\ \nabla v_{h,n+1} \right\|_{2m/(2m-1)} \\
&\qquad +  \|\xi\|_{\rmW^{2,\infty}}\ \sum_{n=\nu}^{N-1} \WW_2^2(u_{h,n+1},u_{h,n}) \,.
\end{align*}
Using the estimates~\eqref{cc7} and~\eqref{cc9} of Lemma~\ref{lem:cc1}, we obtain
$$
\sum_{n=\nu}^{N-1} \WW_2^2(u_{h,n+1},u_{h,n}) \le C\ h\,,
$$
and
\begin{align*}
& h\  \sum_{n=\nu}^{N-1} \left\| \nabla \left(u_{h,n+1}^m \right) - \chi\ u_{h,n+1}\ \nabla v_{h,n+1} \right\|_{2m/(2m-1)}\\
& \qquad\le  \int_{(\nu+1)h}^{(N+1)h} \left\| \nabla \left(u_{h}^m \right) - \chi\ u_h\ \nabla v_h \right\|_{2m/(2m-1)} \dd s \\
& \qquad\le  C\ [1+(N+1)h]^{(m+1)/2m}\ [(N-\nu)h]^{(m-1)/2m} \\
& \qquad\le  C\ (1+t)^{(m+1)/2m}\ (t-s+h)^{(m-1)/2m} \,,
\end{align*}
so that
\begin{align*}
\left| \int_{\RR^d}  \xi (u_h(t)-u_h(s)) \dd x \right|  & \le  C\ (1+t)^{(m+1)/2m}\ (t-s+h)^{(m-1)/2m}\ \|\nabla\xi\|_{2m} \\
& \qquad +  C\ \|\xi\|_{\rmW^{2,\infty}}\ h \,.
\end{align*}
Since $\HH^d(\RR^d)$ and $\HH^{d/2}(\RR^d)$ are continuously embedded in $\LL^\infty(\RR^d)$ and $\LL^{2m}(\RR^d)$, respectively, we end up with
$$
\left| \int_{\RR^d}  \xi (u_h(t)-u_h(s)) \dd x \right| \le C\ (1+t)^{(m+1)/2m}\ (t-s+h)^{(m-1)/2m}\ \|\xi\|_{\HH^{d+2}}\,, 
$$
from which the result follows by a duality argument.
\end{proof}

\subsection{Convergence}\label{sec:conv2}

We are now in a position to establish the strong convergence of $(u_h)_h$ and $(v_h)_h$ required to identify the equations solved by their cluster points as $h\to 0$. For further use, we define the weight $\varrho$ by
$$
\varrho(x) := \frac{1}{\sqrt{1 + |x|^2 }}\,, \quad x\in\RR^d\,.
$$

\begin{proposition}[Convergence]\label{prop:cc4}
There are a sequence $(h_j)_{j\ge 1}$ of positive numbers in $(0,1)$, $h_j\to 0$, and functions
$$
u\in \mathcal{C}([0,\infty);\HH^{-(d+2)}(\RR^d))\,, \qquad v\in \mathcal{C}([0,\infty);\LL^2(\RR^d;\varrho(x)\dd x))
$$
such that, for all $t>0$, $\delta\in (0,1)$, and $T>1$, 
\begin{equation*}
\begin{array}{lcl}
u_{h_j}(t) \longrightarrow u(t) & \text{ in } & \HH^{-(d+2)}(\RR^d)\,, \\
 & & \\
v_{h_j}(t) \longrightarrow v(t) & \text{ in } & \LL^2(\RR^d;\varrho(x)\dd x)\,, 
\end{array} 
\end{equation*}
\begin{equation}
\begin{array}{lcl}
u_{h_j} \longrightarrow u & \text{ in } & \LL^2(\delta,T;\LL^m(\RR^d))\,, \\
 & & \\
v_{h_j} \longrightarrow v & \text{ in } & \LL^2(\delta,T;\HH^1(\RR^d;\varrho(x)\dd x))\,.
\end{array} \label{cc17}
\end{equation}
\end{proposition}

\begin{proof} 
Since $\HH^1(\RR^d)$ is continuously embedded in $\LL^2(\RR^d;\varrho(x)\dd x)$, it follows from the estimates~\eqref{cc7} and~\eqref{cc6} of Lemma~\ref{lem:cc1} that $\{ v_h(t)\ :\ (t,h)\in [0,\infty)\times (0,1) \}$ lies in a compact subset of $\LL^2(\RR^d;\varrho(x)\dd x)$ while the time equicontinuity of Lemma~\ref{lem:cc3} entails that
$$
\limsup_{h\to 0} \int_{\RR^d} |v_h(t,x)-v_h(s,x)|^2\ \varrho(x)\ \dd x \le C_7\ \sqrt{t-s}\,, \qquad 0\le s < t\,.
$$
We then infer from a refined version of the Ascoli-Arzel\`a theorem, see \cite[Proposition~3.3.1]{AGS08}, that there are a sequence $(h_j)_{j\ge 1}$ of positive numbers in $(0,1)$, $h_j\to 0$, and $v\in \mathcal{C}([0,\infty);\LL^2(\RR^d;\varrho(x)\dd x))$ such that, for all $t\ge 0$, 
\begin{equation*}
v_{h_j}(t) \longrightarrow v(t) \;\;\text{ in }\;\; \LL^2(\RR^d;\varrho(x)\dd x)\,. 
\end{equation*}
It next readily follows from the estimate~\eqref{cc6} of Lemma~\ref{lem:cc1} and Lebesgue's dominated convergence theorem that 
\begin{equation}
v_{h_j} \longrightarrow v \;\;\text{ in }\;\; \LL^p(0,T;\LL^2(\RR^d;\varrho(x)\dd x)) \label{cc19}
\end{equation}
for all $p\in [1,\infty)$ and $T>0$.

Furthermore, recall that $\HH^2(\RR^d)$ is compactly embedded in $\HH^1(\RR^d;\varrho(x)\dd x)$ which is continuously embedded in $\LL^2(\RR^d;\varrho(x)\dd x)$. Given $\delta\in (0,1)$ and $T>1$, $(v_h)_h$ is bounded in $\LL^2(\delta,T;\HH^2(\RR^d))$ by Corollary~\ref{cor:cc2}. We may then apply a standard compactness result, see \cite[Lemma~9]{Si87}, to deduce from \eqref{cc19} the second convergence stated in \eqref{cc17}.

Next, $\HH^{1/m}(\RR^d)$ is densely and continuously embedded in $\LL^{m/(m-1)}(\RR^d)$. Consequently, $\LL^m(\RR^d)$ is continuously embedded in $H^{-1/m}(\RR^d)$ and we then conclude that $\LL^m(\RR^d)\cap \LL^1(\RR^d; (1+|x|^2)\dd x)$ is compactly embedded in $\HH^{-(d+2)}(\RR^d)$. Owing to the estimates~\eqref{cc7} and~\eqref{cc6} of Lemma~\ref{lem:cc1}, there is for each $T>0$ a compact subset of $\HH^{-(d+2)}(\RR^d)$ (depending on $T$) in which $u_h(t)$ lies for all $h\in (0,1)$ and $t\in [0,T]$. This fact along with the time equicontinuity of $(u_h)$ in Lemma~\ref{lem:cc3} and \cite[Proposition~3.3.1]{AGS08}, ensure that there are a subsequence of $(h_j)_{j\ge 1}$ (not relabelled) and a function $u\in \mathcal{C}([0,\infty);\HH^{-(d+2)}(\RR^d))$ such that, for all $t\ge 0$, 
\begin{equation}
u_{h_j}(t) \longrightarrow u(t) \;\;\text{ in }\;\; \HH^{-(d+2)}(\RR^d)\,. \label{cc20}
\end{equation}

Thanks to the estimate~\eqref{cc7} of Lemma~\ref{lem:cc1} and the continuous embedding of $\LL^m(\RR^d)$ in $\HH^{-(d+2)}(\RR^d)$, we may combine \eqref{cc20} and Lebesgue's dominated convergence theorem to deduce that 
\begin{equation}
u_{h_j} \longrightarrow u \;\;\text{ in }\;\; \LL^p(0,T;\HH^{-(d+2)}(\RR^d)) \label{cc21}
\end{equation}
for all $p\in [1,\infty)$ and $T>0$.

Finally, given $\delta\in (0,1)$ and $T>0$, it follows from the estimates~\eqref{cc7}--\eqref{cc6} of Lemma~\ref{lem:cc1} and Corollary~\ref{cor:cc2} that 
\begin{equation}
(u_{h_j})_{j\ge 1} \;\;\mbox{ is bounded in }\;\; \LL^2(\delta,T;\rmW^{1,m}(\RR^d))\cap \LL^1(\RR^d; (1+|x|^2)\dd x))\,. \label{cc22}
\end{equation}
Since $\rmW^{1,m}(\RR^d)\cap \LL^1(\RR^d; (1+|x|^2)\dd x)$ is compactly embedded in $\LL^m(\RR^d)$ and $\LL^m(\RR^d)$ is continuously embedded in $\HH^{-(d+2)}(\RR^d)$, a further use of~\cite[Lemma~9]{Si87} allows us to deduce from~\eqref{cc21} and~\eqref{cc22} the first convergence stated in~\eqref{cc17}.
\end{proof}

\subsection{Identifying the limit}\label{sec:itl}

In addition to the convergence of Proposition~\ref{prop:cc4}, we note that the estimates of Corollary~\ref{cor:cc2} ensure that, after possibly extracting a further subsequence, we may assume that
\begin{eqnarray}
& & \Delta v_{h_j} \rightharpoonup \Delta v \;\;\text{ in }\;\; \LL^2((\delta,T)\times \RR^d)\,, \label{cc23} \\
& & u_{h_j} \rightharpoonup u \;\;\text{ in }\;\; \LL^{2m}(\delta,T; \LL^2(\RR^d))\,. \label{cc24}
\end{eqnarray}
Now, checking that $(u,v)$ is a weak solution to \eqref{PKS} follows from Lemma~\ref{lem:b7}, Proposition~\ref{prop:cc4}, \eqref{cc23}, and \eqref{cc24} by classical arguments. Observe in particular that~\eqref{cc17} and~\eqref{cc24} imply that $\left( u_{h_j} \nabla v_{h_j} \right)_{j\ge 1}$ converges weakly towards $u\nabla v$ in $\LL^1((\delta,T)\times\RR^d; \varrho(x)\dd x)$ for all $\delta\in (0,1)$ and $T>1$. We have thus constructed a weak solution to the parabolic-parabolic Keller-Segel system~\eqref{resPKS} and proved Theorem~\ref{main}.

\section*{Acknowledgements} 
Part of this work was done while the authors enjoyed the hospitality and support of the Centro de Ciencias Pedro Pascuale de Benasque.

\bigskip\noindent{\small \copyright\, 2012 by the authors. This paper is under the Creative Commons licence Attribution-NonCommercial-ShareAlike 2.5.}
%

%
\end{document}